\documentclass[12pt,a4paper,reqno]{amsart}
\usepackage{amsfonts, amsmath, amssymb, amsthm, amsbsy, amscd, euscript}

\newtheorem{theor}{Theorem}
\newtheorem{prop}{Proposition}
\newtheorem{cor}{Corollary}
\newtheorem{lemma}{Lemma}

\theoremstyle{definition}
\newtheorem{example}{Example}
\newtheorem{define}{Definition}

\theoremstyle{remark}
\newtheorem{rem}{Remark}

\newcommand{\pinner}{\mathbin{\mathchoice
   {\hbox{\vrule width0.6em depth0pt height0.4pt
   \vrule width0.4pt depth0pt height0.8ex}}
   {\hbox{\vrule width0.6em depth0pt height0.4pt
   \vrule width0.4pt depth0pt height0.8ex}}
   {\hbox{\kern0.14em
   \vrule width0.48em depth0pt height0.4pt
   \vrule width0.4pt depth0pt height0.6ex\kern0.14em}}
   {\hbox{\kern0.1em
   \vrule width0.39em depth0pt height0.4pt
   \vrule width0.4pt depth0pt height0.5ex\kern0.1em}}}}

\newcommand{\what}{\widehat}
\DeclareMathOperator{\img}{im}
\DeclareMathOperator{\id}{id}

\DeclareMathOperator{\Ann}{Ann}

\DeclareMathOperator{\D}{D}

\DeclareMathOperator{\ad}{ad}
\DeclareMathOperator{\Mor}{Mor}
\DeclareMathOperator{\Hom}{Hom}
\DeclareMathOperator{\CDiff}{\mathcal{C}Diff}
\DeclareMathOperator{\Alt}{Alt}
\DeclareMathOperator{\const}{const}
\DeclareMathOperator{\dvol}{dvol}

\newcommand{\ov}{\overline}

\newcommand{\bbf}{\boldsymbol{f}}
\newcommand{\cE}{\mathcal{E}}
\newcommand{\cEinf}{\mathcal{E}^{\infty}}
\newcommand{\MC}{\text{\textup{MC}}}
\newcommand{\cH}{\mathcal{H}}
\newcommand{\cL}{\mathcal{L}}
\newcommand{\cC}{\mathcal{C}}
\newcommand{\cF}{\mathcal{F}}
\newcommand{\cX}{{\EuScript X}}
\newcommand{\cY}{{\EuScript Y}}
\newcommand{\cZ}{{\EuScript Z}}
\newcommand{\veps}{\varepsilon}
\newcommand{\BBR}{\mathbb{R}}

\newcommand{\BBZ}{\mathbb{Z}}

\newcommand{\bp}{\boldsymbol{p}}
\newcommand{\dd}{\partial}

\newcommand{\revd}{\vec{\dd}} 

\newcommand{\rdelta}{\vec{\delta}} 
\newcommand{\Id}{{\mathrm d}}
\newcommand{\fg}{\mathfrak{g}}

\newcommand{\bunit}{\boldsymbol{1}}

\newcommand{\by}[1]{\textit{{#1}}}
\newcommand{\jour}[1]{\textit{{#1}}}
\newcommand{\vol}[1]{\textbf{{#1}}}
\newcommand{\book}[1]{\textrm{{#1}}}

\newcommand{\lshad}{[\![}
\newcommand{\rshad}{]\!]}

\title{Non\/-\/Abelian Lie algebroids over jet spaces}

\author[A.~V.~Kiselev]{A.~V.~Kiselev${}^{\S}$}
\thanks{${}^{\S}$
  \textit{Address}:
  Johann Bernoulli Institute for Mathematics and Computer Science,
  University of Groningen,
  P.O.Box~407, 9700\,AK Groningen, The Netherlands.\quad
  \textit{E-mail}: \texttt{A.V.Kiselev\symbol{"40}rug.nl}%
}

\author[A.~O.~Krutov]{A.~O.~Krutov$^{\dag}$}  
\thanks{${}^{\dag}$%
  \textit{Address}: %
  Department of Higher Mathematics, Ivanovo State Power
  University, Rabfa\-kov\-skaya str.~34, Ivanovo, 153003 Russia.
\quad \textit{E-mail}:
\texttt{krutov\symbol{"40}math.ispu.ru}
}

\date{November 13, 2013, revised January 6, 2014}


\begin{document}

\dedicatory{This submission to proceedings of the jubilee workshop `Nonlinear
Mathematical Physics\textup{:}\\ Twenty Years of JNMP' 
\textup{(}June \textup{4\/--\/14, 2013;} Sophus Lie Centre, Nordfj\o rdeid, 
Norway\textup{)}\\ is a tribute to the legacy of great Norwegian 
mathematicians\textup{:} Niels 
Abel and Sophus Lie.}

\begin{abstract}
We associate Hamiltonian homological evolutionary vector fields 
--\,which are the non\/-\/Abelian variational Lie algebroids' differentials\,--
with Lie algebra\/-\/valued zero\/-\/curvature representations for partial differential equations. 
\end{abstract}

\maketitle
\thispagestyle{empty}\enlargethispage{\baselineskip}
\subsection*{Introduction}
Lie algebra\/-\/valued zero\/-\/curvature representations for partial differential equations (PDE) are the input data for solving Cauchy's problems by the inverse scattering method~\cite{ZSh}. 
For a system of PDE with 
unknowns in two independent variables 
to be kinematically integrable, a zero\/-\/curvature representation at hand must 
depend on a spectral parameter which is non\/-\/removable under gauge transformations. 
In the paper~\cite{Marvan2002} M.~Mar\-van developed a remarkable method for inspection 
whether a parameter in a given zero\/-\/curvature representation~$\alpha$ is 
(non)\/removable; this technique refers to a cohomology theory generated by 
a differential~$\boldsymbol{\partial}_{\alpha}$, 
which was explicitly constructed for every~$\alpha$.  

In this paper we show that zero\/-\/curvature representations for PDE give rise to a 
natural class of non\/-\/Abelian variational Lie algebroids. 
In section~\ref{secPre} (see Fig.\,\ref{FigNonAbelAlgd} 
on p.\,\pageref{FigNonAbelAlgd}) we list all the components of such structures 
(cf.~\cite{KiselevTMPh2011}); in particular, we show that Mar\-van's operator~$
\boldsymbol{\partial}_{\alpha}$ is the anchor. 
In section~\ref{secNonAbelian}, non\/-\/Abelian variational Lie algebroids are 
realized via BRST\/-\/like homological evolutionary vector fields~$Q$ on 
superbundles \`a\ la~\cite{BRST}.
Having enlarged the BRST-\/type setup to a geometry which goes in a complete parallel 
with the standard BV-\/zoo~(\cite{BV1981-83}, see also~\cite{AKZS}), 
in section~\ref{secMaster} we extend the vector field~$Q$ to the evolutionary 
derivation $\what{Q}(\cdot) \cong \lshad \what{S}, \,\cdot \rshad$ whose Hamiltonian 
functional~$\what{S}$ satisfies the classical master\/-\/equation 
$\lshad \what{S}, \what{S} \rshad = 0$. We then address that equation's gauge symmetry 
invariance and 
$\what{Q}$-\/cohomology automorphisms (\cite{KontsevichSoibelman},
cf.\,\cite{FelderKazhdanCME12} and~\cite{KKIgonin2003}),
which yields the next generation of Lie algebroids,
see Fig.\,\ref{FigReiterate} on p.\,\pageref{FigReiterate}.

Two appendices follow the main exposition. We first recall the notion of Lie algebroids over usual smooth manifolds. (Appendix~\ref{AppLieAlgd} concludes with an elementary explanation why the classical construction stops working over infinite jet spaces or over~PDE such as gauge systems.) Secondly, we describe the idea of parity\/-\/odd neighbours to vector spaces and their use in $\BBZ_2$-\/graded superbundles~\cite{Voronov2002}. In particular, we recall how Lie algebroids or Lie algebroid differentials are realised in terms of homological vector fields on the total spaces of such superbundles~\cite{Vaintrob}.

In the earlier work~\cite{KiselevTMPh2011} by the first author and J.\,W.\,van de~Leur,
classical notions, operations, and reasonings which are contained in both appendices were upgraded from ordinary manifolds to jet bundles, which are endowed with their own, 
restrictive geometric structures such as the Cartan connection~$\nabla_{\mathcal{C}}$
and which harbour systems of~PDE. We prove now that the geometry of Lie algebra\/-\/valued connection $\fg$-\/forms~$\alpha$ satisfying zero\/-\/curvature equation~\eqref{eqMC} gives rise to the geometry of solutions~$\widehat{S}$ for the classical master\/-\/equation
\begin{align}
\cE_{\text{CME}}&=\bigl\{
\boldsymbol{i}\hbar\,\smash{\Delta\widehat{S}{\bigr|}_{\hbar=0}} = 
\tfrac{1}{2}\lshad\widehat{S},\widehat{S}\rshad
\bigr\},\label{EqCMEIntro}\\
\intertext{see Theorem~\ref{Th34} on p.~\pageref{Th34} below. 
It is readily seen that realization~\eqref{EqCMEIntro} of the gauge\/-\/invariant setup
is the classical limit of the full quantum picture as $\hbar\to0$;
the objective of quantization $\widehat{S}\longmapsto S^\hbar$ 
is a solution of the quantum master\/-\/equation}
\cE_{\text{QME}}&=\bigl\{
\boldsymbol{i}\hbar\,\Delta S^\hbar =
\smash{\tfrac{1}{2}}\lshad S^\hbar, S^\hbar\rshad
\bigr\} \label{QME}
\end{align}
for the true action functional~$S^\hbar$ at $\hbar\neq0$. Its construction involves quantum,
noncommutative objects such as the deformations~$\fg_\hbar$ of Lie algebras together with deformations of their duals (cf.~\cite{DrinfeldICM86}).
(In fact, we express 
the notion
of non\/-\/Abelian variational Lie algebroids in terms of the homological
evolutionary vector field~$\what{Q}$ and classical
master\/-\/equation~\eqref{EqCMEIntro} viewing this construction as an
intermediate step 
towards quantization.)
A transition from the semiclassical to quantum picture results in 
$\fg_\hbar$-\/va\-lued connections, quantum gauge groups, quantum vector spaces 
for values of the wa\-ve functions in auxiliary linear problems~\eqref{EqAuxLinPrb}, 
and quantum extensions of physical fields.\footnote{Lie algebra\/-\/valued
connection one\/-\/forms are the main objects in classical gauge field theories.
Such physical models are called \emph{Abelian} --\,e.g., 
Maxwell's electrodynamics\,--
or \emph{non\/-\/Abelian} --\,here, consider the Yang\/--\/Mills theories with structure
Lie groups $SU(2)$ or~$SU(3)$\,-- according to the commutation table for the underlying
Lie algebra. This is why we say that variational Lie algebroids are 
(\emph{non}-)\/\emph{Abelian}--- referring to the Lie algebra\/-\/valued connection
one\/-\/forms~$\alpha$ in the geometry of gauge\/-\/invariant zero\/-\/curvature 
representations for~PDE.}
 

\section{Preliminaries}\label{secPre}
\noindent%
Let us first briefly recall some definitions 
(see~\cite{BVV,GDE2012,Olver} and~\cite{Marvan2002} for detail);
this material is standard so that we now fix the notation.

\enlargethispage{0.7\baselineskip}
\subsection{The geometry of infinite jet space $J^\infty(\pi)$}
Let $M^n$~be a smooth real $n$-\/dimensional orientable
manifold. Consider a smooth 
vector bundle
$\pi\colon E^{n+m} \to M^n$ with $m$-\/dimensional fibres and construct the 
space~$J^{\infty}(\pi)$ of infinite jets of sections for~$\pi$. 
A convenient organization of local coordinates is as follows: 
let $x^i$~be some coordinate system on a chart in the base~$M^n$ and 
denote by~$u^j$ the coordinates 
along a 
fibre of the bundle~$\pi$ so that the variables~$u^j$ play the r\^o\-le of unknowns;
one obtains the collection $u^j_\sigma$ of jet variables along fibres of the vector bundle $J^{\infty}(\pi)\to M^n$ 
(here $|\sigma|\geqslant0$ and~$u^j_\varnothing\equiv u^j$). In this setup, 
the \emph{total derivatives} $D_{x^i}$~are commuting vector fields
$D_{x^i}= \nabla_{\mathcal{C}}(\dd/\dd x^i)= \partial / \partial x^i +
\sum_{j,\sigma}u^j_{\sigma i}\,\partial / \partial u^j_\sigma$
on~$J^\infty(\pi)$.

Consider a system of partial differential equations
\[
\cE = \left\{ F^\ell (x^i, u^j, \dots, u^j_\sigma, \dots) = 0, \quad \ell =   1,\dots, r<\infty \right\};
\]
without any loss of generality for applications we assume that the system at hand satisfies mild assumptions
which are outlined in~\cite{GDE2012,Olver}. Then the system~$\cE$ and all its differential consequences $D_{\sigma}(F^\ell) =
0$ (thus presumed existing, regular, and not leading to any contradiction in the course of derivation) generate the infinite prolongation~$\cEinf$ of the system~$\cE$.

Let us denote by~$\bar{D}_{x^i}$ the restrictions of total derivatives~$D_{x^i}$ to~$\cEinf\subseteq J^\infty(\pi)$. 
We recall that the vector fields~$\bar{D}_{x^i}$ span the Cartan distribution~$\cC$ in the tangent space~$T\cEinf$.
At every point $\theta^\infty\in\cEinf$ the tangent 
space~$T_{\theta^\infty}\cEinf$ splits in a direct sum of two subspaces. The one which is spanned by the Cartan distribution~$\cEinf$ is \emph{horizontal} and the other is \emph{vertical}:
$T_{\theta^\infty} \cEinf = \cC_{\theta^\infty} \oplus
V_{\theta^\infty} \cEinf$. 
We denote by~$\Lambda^{1,0}(\cEinf) = \Ann \cC$ 
and~$\Lambda^{0,1} (\cEinf) = \Ann V\cEinf$ 
the $C^\infty(\cEinf)$-\/modules of contact and horizontal
one\/-\/forms which vanish on~$\cC$ and~$V\cEinf$, respectively.
Denote further by~$\Lambda^r(\cEinf)$ the 
$C^{\infty}(\cEinf)$-\/module of $r$-forms on~$\cEinf$.
There is a natural decomposition $\Lambda^r(\cEinf) = \bigoplus_{q+p = r}
\Lambda^{p,q}(\cEinf)$, where $\Lambda^{p,q} (\cEinf) = \bigwedge^p \Lambda^{1,0}(\cEinf) \wedge \bigwedge^q \Lambda^{0,1}(\cEinf)$. This implies that the de Rham differential~$\bar{\Id}$ on~$\cEinf$ is subjected to the decomposition $\bar{\Id} = \bar{\Id}_h + \bar{\Id}_{\cC}$, where $\bar{\Id}_h \colon \Lambda^{p,q}(\cEinf) \to \Lambda^{p,q+1}(\cEinf)$ is the horizontal differential and $\bar{\Id}_{\cC} \colon \Lambda^{p,q}(\cEinf) \to \Lambda^{p+1,q}(\cEinf)$ is the vertical differential.
In local coordinates, the differential~$\bar{\Id}_h$ acts by the rule
\[
\bar{\Id}_h = \sum\nolimits_i \Id x^i \wedge \bar{D}_{x^i}.
\]
We shall use this formula in what follows.
By definition, we put $\bar{\Lambda}(\cE^\infty)=\bigoplus_{q\geqslant0}\Lambda^{0,q}
(\cE^\infty)$ and we denote by $\overline{H}^n(\cdot)$ the senior $\Id_h$-\/cohomology
groups (also called senior \emph{horizontal cohomology}) for the infinite jet bundles
which are indicated in parentheses, cf.~\cite{gvbv}.

\begin{rem}\label{RemNoExcess}
The geometry which we analyse in this paper is produced and arranged by using the pull-backs $f^*(\varrho)$ of fibre 
bundles $\varrho$ under some mappings $f$. Typically, the fibres of $\varrho$ are Lie algebra-valued horizontal 
differential forms coming from $\Lambda^*(M^n)$, or similar objects%
\footnote{%
Let us specify at once that the geometries of prototype fibres in the bundles under study are described by $\mathfrak{g}$-,
$\mathfrak{g}^*$-, $\Pi\mathfrak{g}$-, or $\Pi\mathfrak{g}^*$-valued $(-1)$-, zero-, one-, two-, and three-forms\,;
the degree $-1$ corresponds to the module $D_1(M^n)$ of vector fields.}\,;
in turn, the mappings $f$ are projections to the base $M^n$ of some infinite jet bundles. We employ the standard notion of
\emph{horizontal infinite jet bundles} such as $\overline{J^{\infty}_{\xi}}(\chi)$ or $\overline{J^{\infty}_{\chi}}(\xi)$
over infinite jet bundles $J^{\infty}(\xi)$ and $J^{\infty}(\chi)$, respectively\,; these spaces are present in
Fig.~\ref{FigNonAbelAlgd} on p.~\pageref{FigNonAbelAlgd} and they occur in (the proof of) 
Theorems~\ref{propNonAbelianQ} and~\ref{Th34}
below. A proof of the convenient isomorphism 
$\overline{J^{\infty}_{\xi}}(\chi)\cong J^{\infty}(\xi\mathbin{\times_{M^n}}\chi)=J^{\infty}(\xi)\mathbin{\times_{M^n}}J^{\infty}(\chi)$
is written in~\cite{RingersProtaras}, see also references therein. However, we recall further that, strictly speaking,
the entire picture --~with fibres which are inhabited by form-valued parity-even or parity-odd (duals of the) Lie algebra
$\mathfrak{g}$~-- itself is the image of a pull-back under the projection
$\pi_{\infty}\colon J^{\infty}(\pi)\to M^n$ in the infinite jet bundle over the bundle $\pi$ of physical fields. In other
words, \emph{sections} of those induced bundles are elements of Lie algebra etc., but all coefficients are differential
functions in configurations of physical fields (which is obvious, e.\,g., from~\eqref{eqMC} in Definition~\ref{DefZCR}
on the next page). Fortunately, it is the composite geometry of a fibre but not its location over the composite-structure
base manifold which plays the main r\^ole in proofs of Theorems~\ref{propNonAbelianQ} and~\ref{Th34}.

It is clear now that an attempt to indicate not only the bundles $\xi$ or $\chi,\ \Pi\chi^*,\ \Pi\xi$, and $\xi^*$ which
determine the intrinsic properties of objects but also to display the bundles that generate the pull-backs would make all
proofs sound like the well-known poem about the house which Jack built.

Therefore, we \emph{denote} the objects such as $p_i$ or $\alpha$ and their mappings (see p.~\pageref{pProofTh1} or
p.~\pageref{pProofTh2}) \emph{as if} they were just sections, $p_i\in\Gamma(\xi)$ and $\alpha\in\Gamma(\chi)$, of the 
bundles $\xi$ and $\chi$ over the base $M^n$, leaving obvious technical details to the reader.
\end{rem}

\subsection{Zero\/-\/curvature representations}
Let $\fg$~be a finite\/-\/dimensional (complex) Lie algebra.
Consider its tensor product (over~$\BBR$) 
with the exterior algebra of horizontal differential
forms~$\bar{\Lambda}(\cEinf)$ on the infinite prolongation of~$\cE$.
This product is endowed with a $\BBZ$-\/graded Lie algebra structure
by the bracket $[A\mu, B\nu] = [A,B]\,\mu\wedge\nu$, where
$\mu,\nu\in\bar{\Lambda}(\cEinf)$ and~$A,B\in\fg$.

Let us focus on the case of $\fg$-valued one\/-\/forms. In the tensor
product, the Jacobi identity for $\alpha$,\ $\beta$,\ $\gamma \in \fg\otimes 
\Lambda^{0,1}(\cEinf)$ looks as follows. Let $\alpha = A \mu$, $\beta =
B\nu$, $\gamma = C \omega $. We obtain that
\begin{align*}
  [\alpha, [\beta, \gamma]] + {}& [\gamma, [\alpha, \beta]] + [\beta,
  [\gamma, \alpha]] \\
{}&{}= 
  [A\mu, [B, C]\, \nu\wedge\omega] + [C\omega, [A,B]\,\mu\wedge\nu] +
  [B\nu, [C, A]\, \omega\wedge\mu] \\
{}&{}  = [A, [B,C]]\,\mu\wedge\nu\wedge\omega + [C,[A,B]]\,\omega\wedge\mu\wedge\nu
  + [B, [C, A]\, \nu\wedge\omega\wedge\mu.
\intertext{For the one\/-\/forms $\mu$,\ $\nu$,\ and~$\gamma$ we have that
$\mu\wedge\nu\wedge\gamma = \gamma\wedge\mu\wedge\nu =
\nu\wedge\gamma\wedge\mu$ so that the above equality continues with}
{}&{}  = \bigl([A, [B, C]] + [C, [A, B]] + [B, [C, A]]\bigr)\,\mu\wedge\nu\wedge\omega
  = 0.
\end{align*}
Indeed, this expression vanishes due to the Jacobi identity  of the
Lie algebra~$\fg$, namely, $[A, [B, C]] + [C, [A,B]] + [B, [C, A]]=0$.

The horizontal differential~$\Id_h$ acts on elements of 
$A\otimes\mu\in\fg\otimes\Lambda(\cEinf)$ as follows: 
\[
\Id_h(A\otimes\mu) = A \otimes\Id_h \mu.
\]

\begin{define}\label{DefZCR}
A horizontal one\/-\/form $\alpha\in\fg\otimes\Lambda^{0,1}(\cEinf)$ is called a $\fg$-\/valued zero\/-\/curvature representation for~$\cE$ if $\alpha$~satisfies the Maurer\/--\/Cartan equation
\begin{equation}\label{eqMC}
\cE_{\text{MC}} = \bigl\{ \bar{\Id}_h \alpha - \tfrac12 [\alpha, \alpha]  
\doteq 0\bigr\}
\end{equation}
by virtue of equation~$\cE$ and its differential consequences.
\end{define}

Given a zero\/-\/curvature representation $\alpha = A_i\,\Id x^i$, 
the Maurer\/--\/Cartan equation~$\cE_{\text{MC}}$ can be interpreted as the compatibility condition for the linear system
\begin{equation}\label{EqAuxLinPrb}
\Psi_{x^i} = A_i \Psi,
\end{equation}
where $A_i\in\fg\otimes C^{\infty}(\cEinf)$ and $\Psi$~is the wave function, that is, $\Psi$~is a (local) section of the principal fibre bundle~$P(\cEinf,G)$ with action of the gauge Lie group~$G$ on fibres; the Lie algebra of~$G$ is~$\fg$.
Then the system of equations
\[
D_{x^i} A_j - D_{x^j} A_i + [A_i, A_j] = 0,\qquad 1\leqslant i < j\leqslant n,
\]
is equivalent to Maurer\/--\/Cartan's equation~\eqref{eqMC}.

\subsection{Gauge transformations}\label{subsecGauge}
Let $\fg$~be the Lie algebra of the Lie group~$G$ and
$\alpha$~be a $\fg$-\/valued zero\/-\/curvature representation
for a given PDE system~$\cE$.
A gauge transformation $\Psi\mapsto g\Psi$ of the wave function by an element~$g
\in C^\infty(\cEinf,G)$ induces the change 
\[
\alpha\mapsto
\alpha^g = g\cdot \alpha \cdot g^{-1} + \bar{\Id}_h g\cdot g^{-1}.
\]
The zero\/-\/curvature representation~$\alpha^{g}$ is called
\emph{gauge equivalent} to the initially given~$\alpha$;
the $G$-\/valued function~$g$ on~$\cEinf$ determines the 
\emph{gauge transformation} of~$\alpha$. For convenience, 
we make no distinction between the gauge transformations $\alpha\mapsto
\alpha^g$ and $G$-\/valued functions~$g$ which generate them.

It is readily seen that a composition of two gauge transformations, 
by using $g_1$ first and then by~$g_2$, 
itself is a gauge transformation generated by the $G$-\/valued function~$g_2\circ g_1$. Indeed, we have that
\begin{multline*}
(\alpha^{g_1})^{g_2} = (\bar{\Id}_h g_1 \cdot g_1^{-1} + g_1 \cdot
\alpha \cdot g_1^{-1})^{g_2} = \bar{\Id}_h g_2 \cdot g_2^{-1} + g_2 \cdot (\bar{\Id}_h
g_1 \cdot g_1^{-1} +  g_1 \cdot \alpha \cdot g_1^{-1})\cdot g_2^{-1} \\
= (\bar{\Id}_h g_2 \cdot g_1 + g_2 \cdot \bar{\Id}_h g_1) \cdot g_1^{-1}
\cdot g_2^{-1} + g_2\cdot g_1\cdot \alpha \cdot g_1^{-1} \cdot
g_{2}^{-1} \\
= \bar{\Id}_h (g_2\cdot g_1) \cdot (g_2\cdot g_1)^{-1} + 
 (g_2\cdot g_1) \cdot \alpha \cdot (g_2\cdot g_1)^{-1}.
\end{multline*}
We now consider \emph{infinitesimal} gauge transformations generated by elements 
of the Lie group~$G$ which are close to its unit element~$\bunit$.
Suppose that $g_1 = \exp(\lambda p_1)=
\bunit + \lambda p_1 + \tfrac{1}{2}\lambda^2 p_1^2+o(\lambda^2)$ and 
$g_2 = \exp(\mu p_2)= \bunit + \mu p_2 + \tfrac{1}{2}\mu^2 p_2^2+o(\mu^2)$ for 
some~$p_1$,\ $p_2\in \fg$ and $\mu$,\ $\lambda\in\BBR$. 
The following lemma, an elementary proof of which refers to the definition of Lie algebra,
is the key to a construction of the anchors in non\/-\/Abelian variational Lie 
algebroids.

\begin{lemma}\label{lemmaGaugeInfCommute}
Let $\alpha$~be a $\fg$-\/valued zero\/-\/curvature representation for a system~$\cE$.
Then the commutant $g_1\circ g_2 \circ g^{-1}_1 \circ g_2^{-1}$ of infinitesimal gauge transformations~$g_1$ and~$g_2$ is 
an infinitesimal gauge transformation again.
\end{lemma}

\begin{proof}
By definition, put $g=g_1\circ g_2\circ g^{-1}_1\circ g_2^{-1}$.
Taking into account that $g_1^{-1}=\bunit-\lambda p_1+\tfrac{1}{2}\lambda^2 p_1^2 
+ o(\lambda^2)$ and $g_2^{-1}=\bunit-\mu p_2+\tfrac{1}{2}\mu^2 p_2^2+ o(\mu^2)$, 
we obtain that
\[
g = g_1g_2g_1^{-1}g_2^{-2} 
=\bunit + \lambda\mu\cdot(p_1p_2-p_2p_1) 
+ o(\lambda^2 + \mu^2).
\]
We finally recall that $[p_1, p_2] \in \fg$, whence follows the assertion.
\end{proof}

An infinitesimal gauge transformation $g =
\bunit + \lambda p + o(\lambda)$ acts on a given $\fg$-valued zero\/-\/curvature representation~$\alpha$ for an equation~$\cEinf$ by the formula
\begin{multline*}
\alpha^g =  \bar{\Id}_h( \bunit + \lambda p + o(\lambda)) \cdot
(\bunit - \lambda p + o(\lambda)) + ( \bunit + \lambda p +
o(\lambda)) \cdot \alpha \cdot ( \bunit - \lambda p + o(\lambda)) \\
= \lambda \bar{\Id}_h p + \alpha + \lambda( p \alpha - \alpha p) +
o(\lambda)
= \alpha + \lambda( \bar{\Id}_h p + [p, \alpha]) + o(\lambda).
\end{multline*}
From the coefficient of~$\lambda$ we obtain the operator
$\bar{\boldsymbol{\partial}}_{\alpha} = \bar{\Id}_h + [ \cdot, \alpha]$. 
Lemma~\ref{lemmaGaugeInfCommute} implies that the image of this operator is closed under commutation in~$\fg$, that is,
$[\img \bar{\boldsymbol{\partial}}_{\alpha}, \img \bar{\boldsymbol{\partial}}_{\alpha} ]
 \subseteq \img \bar{\boldsymbol{\partial}}_{\alpha}$.
Such operators and their properties were studied 
in~\cite{KiselevTMPh2011,KiselevLeur2011}.
We now claim that the operator~$\bar{\boldsymbol{\partial}}_{\alpha}$ yields
the anchor in a non\/-\/Abelian variational Lie algebroid, see Fig.~\ref{FigNonAbelAlgd}; 
\begin{figure}[htb]
{\unitlength=1mm
\special{em:linewidth 0.4pt}
\linethickness{0.4pt}
\begin{picture}(75.00,55)
\put(0.00,5.00){\vector(1,0){25.00}}
\put(12.67,16.67){\vector(0,-1){9.33}}
\put(12.67,5.00){\circle*{1.00}}
\put(12.67,20.00){\vector(0,1){30.00}}
\put(9.33,51.67){\makebox(0,0)[lb]{$\overline{J^{\infty}_{\chi}}(\xi)$}}
\put(14.67,42.67){\makebox(0,0)[lb]{$[\,,\,]_{\mathfrak{g}}$}}
\put(14.67,10.33){\makebox(0,0)[lb]{$\chi_{\infty}\circ\chi^*_{\infty}(\xi)$}}
\put(2.00,5.67){\makebox(0,0)[lb]{$M^n$}}
\put(11.67,1.33){\makebox(0,0)[lb]{$x$}}
\put(21.00,28.33){\vector(1,0){30.00}}
\put(21.33,30.33){\makebox(0,0)[lb]{$\boldsymbol{\dd}_{\alpha}=\Id_h+[\,\cdot\,,\alpha]$}}
\put(49.00,5.00){\vector(1,0){26.00}}
\put(62.33,16.33){\vector(0,-1){9.00}}
\put(62.33,5.00){\circle*{1.00}}
\put(62.33,20.00){\vector(0,1){30.00}}
\put(59.67,22.00){\vector(0,1){8}}
\put(59.67,36.00){\vector(0,1){8}}
\put(65.00,22.00){\vector(0,1){8}}
\put(65.00,36.00){\vector(0,1){8}}
\put(58.67,51.33){\makebox(0,0)[lb]{$\overline{J^{\infty}_{\xi}}(\chi)$}}
\put(69.00,42.00){\makebox(0,0)[lb]{$[\,,\,]$}}
\put(64.00,9.33){\makebox(0,0)[lb]{$\xi_{\infty}\circ\xi^*_{\infty}(\chi)$}}
\put(50.67,5.67){\makebox(0,0)[lb]{$M^n$}}
\put(61.33,1.33){\makebox(0,0)[lb]{$x$}}
\end{picture}}
\caption{Non\/-\/Abelian variational Lie algebroid.}\label{FigNonAbelAlgd}
\end{figure}
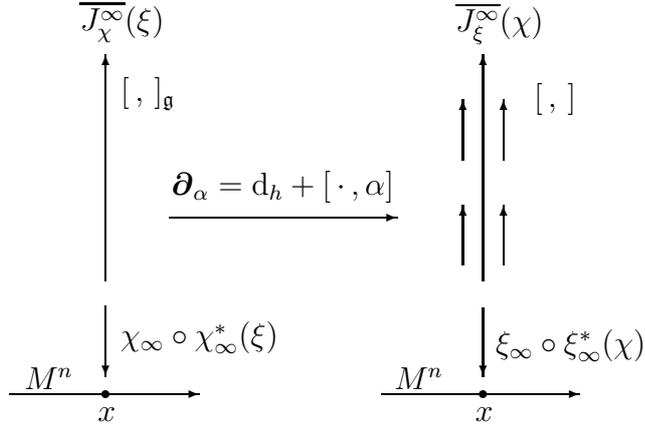
this construction is elementary
(see Remark~\ref{RemNoExcess} on p.~\pageref{RemNoExcess}). Namely, the non-Abelian Lie algebroid
$\left(\pi_{\infty}^*\circ\chi^*_{\infty}(\xi),\boldsymbol{\dd}_{\alpha},[\,,\,]_{\mathfrak{g}}\right)$ 
consists of
\begin{itemize}
\item
the pull-back of the bundle $\xi$ for $\mathfrak{g}$-valued gauge parameters $p$\,; the pull-back is obtained by using the 
bundle $\chi$ for $\mathfrak{g}$-forms $\alpha$ and 
(again by using the infinite jet bundle $\pi_{\infty}$ over)
the bundle $\pi$ of physical fields,
\item
the (restriction $\bar{\boldsymbol{\dd}}_{\alpha}$ to $\cE^{\infty}\subseteq J^{\infty}(\pi)$ of the)
anchor $\boldsymbol{\dd}_{\alpha}$ that generates infinitesimal gauge transformations 
$\dot\alpha=\boldsymbol{\dd}_{\alpha}(p)$ in the bundle $\chi$ of $\mathfrak{g}$-valued connection one-forms, and
\item
the Lie algebra structure $[\,,\,]_{\mathfrak{g}}$ on the anchor's domain of definition.
\end{itemize}
\noindent
We refer to Appendix~\ref{AppLieAlgd} for more detail and 
to p.~\pageref{pDiscussion} for discussion on that object's structural complexity.

\subsection{Noether identities for the Maurer\/--\/Cartan equation}\label{secNoetherIdMC}
In the meantime, let us discuss Noether identities~\cite{BVV,GDE2012,Olver} 
for Maurer\/--\/Cartan equation~\eqref{eqMC}. 
Depending on the dimension~$n$ of the base
manifold~$M^n$, we consider the cases $n=2$, $n=3$, and $n>3$. 
We suppose that the Lie algebra~$\fg$ is equipped\footnote{Notice that the Lie algebra~$\fg$ is canonically
identified with its dual~$\fg^*$ via nondegenerate metric $t_{ij}$.}
with a nondegenerate $\ad$-\/invariant metric~$t_{ij}$.  
The paring $\langle\,,\,\rangle$ is
defined for elements of~$\fg\otimes\Lambda(M^n)$ as follows,
\[
\langle A\mu, B\nu \rangle = \langle A, B \rangle\,\mu \wedge \nu,
\]
where the coupling $\langle A, B \rangle$ is given by the metric~$t_{ij}$
for~$\fg$. From the $\ad$-\/invariance $\langle [A,B], 
C\rangle = \langle A, [B, C] \rangle$ of the metric~$t_{ij}$ we deduce
that 
\begin{multline*}
  \langle [A\mu, B\nu], C\rho \rangle = \langle [A, B]\, \mu \wedge \nu,
  C\rho \rangle  =  \langle [A, B], C \rangle\, \mu\wedge\nu\wedge\rho = 
  \langle A, [B, C] \rangle\, \mu \wedge\nu\wedge\rho \\
  {}= \langle A\mu, [B,C]\, \nu\wedge\rho \rangle  = \langle A \mu, [B
  \nu, C \rho] \rangle.
\end{multline*}
Let us denote by $\cF = -\Id_h \alpha + \tfrac12 [\alpha,\alpha]$
the left\/-\/hand side of Maurer\/--\/Cartan equation~\eqref{eqMC}.
We recall from section~\ref{subsecGauge} that $\dot\alpha
= \boldsymbol{\partial}_{\alpha}(p)$ is a gauge symmetry of
Maurer\/--\/Cartan equation~\eqref{eqMC}. Moreover, for all~$n>1$ the
operator $
\boldsymbol{\dd}_\alpha^\dag
$ produces a Noether
identity for~\eqref{eqMC}, which is readily seen from the following statement.

\begin{prop}\label{propMCNoetherIdenN2}
The left\/-\/hand sides $\cF = -\Id_h\alpha + \tfrac12[\alpha, \alpha]$
of Maurer\/--\/Cartan's equation satisfy the Noether identity 
\textup{(}or \textsl{Bianchi identity} for the curvature two\/-\/form\textup{)}
\begin{equation}\label{eqNoetherId}
\boldsymbol{\dd}_{\alpha}^\dagger(\cF)= - \Id_h \cF - [\cF, \alpha] \equiv 0.
\end{equation}
\end{prop}

\begin{proof}
Applying the operator $\boldsymbol{\dd}_{\alpha}^\dagger$ to the left\/-\/hand sides
of Maurer\/--\/Cartan's equation, we obtain
\begin{multline*}
\boldsymbol{\dd}_{\alpha}^\dagger(\cF) = 
\boldsymbol{\dd}_{\alpha}^\dagger\bigl(-\Id_h \alpha + \tfrac12[\alpha, \alpha]\bigr)= 
(-\Id_h - [\cdot ,\alpha])\bigl(-\Id_h \alpha + \tfrac12[\alpha, \alpha]\bigr)={} \\
{}= (\Id_h\circ\Id_h) \alpha -\tfrac12 \Id_h \bigl([\alpha, \alpha]\bigr) 
+ [\Id_h \alpha, \alpha] - \tfrac12 [\alpha, [\alpha, \alpha]] ={}\\
{}= - [ \Id_h\alpha, \alpha] + [\Id_h \alpha, \alpha] - \tfrac12 [\alpha,
 [\alpha, \alpha]] = 0.    
\end{multline*}
The third term in the last line is zero due to the Jacobi identity,
whereas the first two cancel out.
\end{proof}

Let $n=2$. The Maurer\/--\/Cartan equation's left\/-\/hand sides~$\cF$ 
are top\/-\/degree forms, 
hence every operator 
which increases the form degree vanishes at~$\cF$.  


Consider the case $n=3$; 
we recall that Maurer\/--\/Cartan equation~\eqref{eqMC} is
Euler\/--\/Lagrange in this setup 
(cf.~\cite{TownsendAchucarro1986,AKZS,Witten1988}).

\begin{prop}\label{propMCLagrangian}
If the base manifold $M^3$ is $3$-dimensional, 
then Maurer\/--\/Cartan's equation is Euler\/--\/Lagrange 
with respect to the action functional
\begin{equation}\label{eqMCAction}
S_{\MC} = \int \cL =  \int \left\{ -\tfrac12 \langle \alpha, \Id_h \alpha \rangle
    + \tfrac16 \langle \alpha, [\alpha,\alpha] \rangle \right\}.
\end{equation}
Note that its Lagrangian density~$\cL$ is a well\/-\/defined top\/-\/degree
form on the base threefold~$M^3$.
\end{prop}

\begin{proof}
Let us construct the Euler\/--\/Lagrange equation:
\begin{multline*}
\delta \int\left\{-\tfrac12 \langle \alpha, \Id_h\alpha \rangle + \tfrac16\langle \alpha,
      [\alpha, \alpha] \rangle \right\} 
= \langle \delta\alpha, - \Id_h \alpha\rangle  +
    \tfrac16(\langle \delta \alpha, [\alpha, \alpha] \rangle + \langle
    \alpha, [\delta \alpha, \alpha] \rangle + \langle \alpha, [\alpha,
    \delta \alpha] \rangle \\
    = \langle \delta \alpha, - \Id_h \alpha + \tfrac12 [\alpha, \alpha] \rangle.
\end{multline*}
This proves our claim.
\end{proof}

\begin{prop}
For each $p\in \fg\otimes\Lambda^0(M^3)$, the evolutionary vector
field $\vec{\dd}^{\,(\alpha)}_{A(p)}$ with ge\-ne\-ra\-t\-ing section
$A(p)=\boldsymbol{\dd}_\alpha(p) = \Id_h p + [p, \alpha]$ is a Noether symmetry of the
action~$S_{\MC}$,\footnote{\label{FootNoBoundaty}%
Here $\cong$~denotes the equality up to integration by parts and we assume the absence of boundary terms.}
\[
\vec{\dd}^{\,(\alpha)}_{A(p)}(S_{\MC}) \cong 0 \in \overline{H}^n(\chi).
\]
The operator $A = \boldsymbol{\dd}_\alpha = \Id_h + [\cdot, \alpha]$ determines linear
Noether's identity~\eqref{eqNoetherId},
\[
\Phi(x, \alpha, \cF) =  A^\dag(\cF) \equiv 0,
\]
for left\/-\/hand sides of the system of Maurer\/--\/Cartan's equations~\eqref{eqMC}.
\end{prop}

\begin{proof}
We have
\[
\vec{\dd}^{\,(\alpha)}_{A(p)} S_{\MC} 
 \cong \langle A(p), \tfrac{\delta}{\delta \alpha}S_{\MC} \rangle
  \cong \left\langle \bigl(\ell_{\Phi}^{(\cF)}  \bigr)^\dag(p), \cF
  \right\rangle 
  \cong \langle p, \ell_\Phi^{(\cF)} (\cF)   \rangle 
  = \langle p, \Phi(\cF) \rangle 
  = \langle p, A^\dag(\cF) \rangle.
\]
In Proposition~\ref{propMCNoetherIdenN2} we prove that $A^\dag(\cF) \equiv 0$. 
So for all~$p$ we have that $\langle p, A^\dag(\cF) \rangle \cong 0$, which
concludes the proof. 
\end{proof}

Finally, we let $n>3$.
In this case of higher dimension, the Lagrangian $\cL = \langle
\alpha, \tfrac16 [\alpha, \alpha] - \tfrac12 \Id_h \alpha \rangle\in
\Lambda^3(M^n)$ does not belong to the space of top\/-\/degree
forms 
and Proposition~\ref{propMCLagrangian} does not hold. 
However, Noether's identity $\boldsymbol{\dd}_\alpha^\dagger(\cF)\equiv 0$
still holds if~$n>3$ according to Proposition~\ref{propMCNoetherIdenN2}.

\section{Non\/-\/Abelian variational Lie algebroids}\label{secNonAbelian}
\noindent%
Let $\vec{e}_1$,\ $\ldots$,\ $\vec{e}_d$ be a basis in the Lie algebra~$\fg$. 
Every $\mathfrak{g}$-\/valued zero\/-\/curvature representation for a given PDE 
system~$\cEinf$ is then
$\alpha = \alpha^k_i \vec{e}_k\,\Id x^i$ for some coefficient 
functions~$\alpha^k_i\in C^{\infty}(\cEinf)$.
Construct the vector bundle
$\chi \colon \Lambda^1(M^n)\otimes
\mathfrak{g} \to M^n$ 
and the trivial bundle $\xi \colon M^n\times\fg\to M^n$ with the Lie algebra~$\fg$ 
taken for fibre. 
Next, introduce the superbundle $\Pi\xi\colon M^n\mathbin{\times} \Pi\mathfrak{g}\to M^n$ 
the total space of which is the same as that of~$\xi$ but such that 
the parity of fibre coordinates is reversed\footnote{The odd neighbour~$\Pi\fg$ of the Lie algebra is introduced in order to handle poly\/-\/linear, totally skew\/-\/symmetric maps of elements of~$\fg$ so that the parity\/-\/odd space~$\Pi\fg$ carries the information about the Lie algebra's structure constants $c^k_{ij}$ still not itself becoming a Lie superalgebra.}
(see Appendix~\ref{AppOdd} on p.~\pageref{AppOdd}).
Finally, consider the Whitney sum
$J^{\infty}(\chi)\times_{M^n}J^{\infty}(\Pi\xi)$ 
of infinite jet bundles over the parity\/-\/even vector bundle~$\chi$ and parity\/-\/odd~$\Pi\xi$.

With the geometry of every $\mathfrak{g}$-\/valued zero\/-\/curvature representation 
we associate a non\/-\/Abelian variational Lie algebroid~\cite{KiselevTMPh2011}. 
Its realization by a homological evolutionary vector field is the differential in the arising gauge cohomology theory (cf.~\cite{Vaintrob} and~\cite{AKZS,KKIgonin2003,KiselevTMPh2011,KontsevichSoibelman,Marvan2002}).

\begin{theor}\label{propNonAbelianQ}
The parity\/-\/odd evolutionary vector field which encodes 
the non\/-\/Abelian variational Lie algebroid structure 
on the infinite jet superbundle $J^\infty(\chi\mathbin{{\times}_{M^n}}\Pi\xi)\cong
J^\infty(\chi)\mathbin{{\times}_{M^n}}J^\infty(\Pi\xi)$ is 
\begin{equation}\label{eqQBRST}
Q = \vec{\partial}^{\,(\alpha)}_{[b,\alpha] + \Id_h b} +
  \tfrac12 \vec{\partial}^{\,(b)}_{[b,b]}, \qquad [Q, Q] =0 \quad
  \Longleftrightarrow \quad Q^2 =0,
\end{equation}
where for each choice of respective indexes,
\begin{itemize}
\item 
$\alpha^k_\mu$ is 
a parity\/-\/even coordinate along fibres in the bundle~$\chi$ of $\mathfrak{g}$-\/valued one\/-\/forms,
\item 
$b^k$ is a parity\/-\/odd fibre coordinate in the bundle~$\Pi\xi$,
\item 
$c^k_{ij}$ is a structure constant in the Lie algebra $\mathfrak{g}$ so that 
$[b^i,b^j]^k = b^i c^k_{ij} b^j$ and $[b^i,\alpha^j]^k = b^i c^k_{ij} \alpha^j$,
\item
$\Id_h$ is the horizontal differential on the Whitney sum of infinite jet bundles,
\item
the operator $\boldsymbol{\partial}_{\alpha} = \Id_h + [\cdot,\alpha] \colon
\overline{J^\infty_\chi}(\Pi\xi)\cong J^\infty(\chi\mathbin{{\times}_{M^n}}\Pi\xi) \to
\overline{J^\infty_{\Pi\xi}}(\chi)\cong J^\infty(\chi\mathbin{{\times}_{M^n}}\Pi\xi)$
is the anchor.
\end{itemize}
\end{theor}

\begin{proof}
The anticommutator $[Q,Q] = 2Q^2$ of the parity\/-\/odd vector field~$Q$ with itself is again an evolutionary vector field.
Therefore it suffices to prove that the coefficients of 
$\vec{\partial}/\partial \alpha$ and $\vec{\partial}/\partial b$ are equal to zero in the vector field
\[
Q^2 = \left(\vec{\partial}^{\,(\alpha)}_{[b,\alpha] + \Id_h b} +
  \tfrac12 \vec{\partial}^{\,(b)}_{[b,b]}\right)
\left(\vec{\partial}^{\,(\alpha)}_{[b,\alpha] + \Id_h b} +
  \tfrac12 \vec{\partial}^{\,(b)}_{[b,b]}\right).
\]
We have $[b,b]^k = b^i c_{ij}^k b^j$ by definition. Hence it is readily seen that 
$(\tfrac12 \vec{\partial}^{\,(b)}_{b^i c^k_{ij} b^j})^2 = 0$ 
because $\mathfrak{g}$~is a Lie algebra~\cite{Voronov2002} 
so that the Jacobi identity is satisfied by the structure constants.
Since the bracket $[b,b]$ does not depend on~$\alpha$, we deduce that
$ (\vec{\partial}^{\,(\alpha)}_{[b,\alpha] + \Id_h b}) (\tfrac12
\vec{\partial}^{\,(b)}_{[b,b]}) = 0$. Therefore,
\begin{multline*}
    Q^2 = \left(\vec{\partial}^{\,(\alpha)}_{[b,\alpha] + \Id_h b} +
      \tfrac12 \vec{\partial}^{\,(b)}_{[b,b]}\right)
    \left(\vec{\partial}^{\,(\alpha)}_{[b,\alpha] + \Id_h b} \right) =
    - \vec{\partial}^{\,(\alpha)}_{[b,[b,\alpha] + \Id_h b]} +
    \tfrac12 \vec{\partial}^{\,(\alpha)}_{[[b,b],\alpha]+\Id_h([b,b])}\\
    = \vec{\partial}^{\,(\alpha)}_{ -[b,[b,\alpha] + \Id_h b] +
      \frac12[[b,b],\alpha]+\frac12\Id_h([b,b])}.
\end{multline*}
Now consider the expression $-[b,[b,\alpha] + \Id_h b]
+\tfrac12[[b,b],\alpha]+\tfrac12\Id_h([b,b])$, viewing it as a bi\/-\/linear 
skew\/-\/symmetric map \label{pProofTh1}
$\Gamma(\xi)\times\Gamma(\xi)\to\Gamma(\chi)$. 
First, we claim that the value 
$\bigl( \tfrac12[[b,b],\alpha] - [b,[b,\alpha]] \bigr)(p_1,p_2)$
at any two sections $p_1$,\ $p_2\in \Gamma(\xi)$ vanishes identically. Indeed, 
by taking an alternating sum over the permutation group of two elements we have that
\begin{multline*}
 \tfrac12[[p_1,p_2],\alpha] - \tfrac12[[p_2,p_1],\alpha] -
 [p_1,[p_2,\alpha]] + [p_2,[p_1,\alpha]] = [[p_1,p_2],\alpha] -
 [p_1,[p_2,\alpha]] - [p_2,[\alpha,p_1]] \\= - [\alpha, [p_1,p_2]]
    - [p_1,[p_2,\alpha]] - [p_2,[\alpha,p_1] = 0.
\end{multline*}
At the same time, the value of bi\/-\/linear skew\/-\/symmetric mapping 
$\tfrac12 \Id_h([b,b]) - [b,\Id_h b]$ at sections $p_1$ and $p_2$ also vanishes,
\[
\tfrac12 \Id_h([p_1,p_2]) - \tfrac12 \Id_h([p_2,p_1]) -
[p_1, \Id_h p_2] + [p_2, \Id_h p_1] = \Id_h([p_1, p_2])
- [p_1, \Id_h p_2] - [\Id_h p_1, p_2] = 0.
\]
We conclude that 
\[
Q^2{\Bigr|}_{(p_1,p_2)} 
= \vec{\partial}^{\,(\alpha)}_{ \left\{-[b,[b,\alpha] + \Id_h b] +
    \frac12[[b,b],\alpha]+\frac12\Id_h([b,b])\right\}(p_1,p_2) }
  = \vec{\partial}^{\,(\alpha)}_{0} = 0,
  \]
which proves the theorem.
\end{proof}

Finally, let us derive a reparametrization formula for the homological vector field~$Q$ 
in the course of gauge transformations of zero\/-\/curvature representations.
We begin with some trivial facts~\cite{Bourbaki,DNF}.

\begin{lemma}
Let $\alpha$~be a $\fg$-valued zero\/-\/curvature representation for a PDE system. Consider two infinitesimal gauge transformations 
given by $g_1 = \bunit +\veps p_1 + o(\veps)$ and $g_2= \bunit + \veps p_2 + o(\veps)$. 
Let $g\in C^\infty(\cEinf,G)$ also determine a gauge transformation. 
Then the following diagram is commutative,
\[\begin{CD}
  \alpha^g @>{g_2}>> \beta \\
  @AA{g}A  @AA{g}A \\
  \alpha @>{g_1}>> \alpha^{g_1},
\end{CD}\]
if the relation $p_2 = g \cdot p_1 \cdot g^{-1}$ is valid.
\end{lemma}

\begin{proof}
By the lemma's assumption we have that $(\alpha^{g_1})^g = (\alpha^{g})^{g_2}$. 
Hence we deduce~that
\[
g\cdot(\bunit + \veps p_1) =  (\bunit + \veps p_2)\cdot g 
\qquad\Longleftrightarrow\qquad
g\cdot p_1 =  p_2 \cdot g,
\]
which yields the transformation rule
$p_2 =  g\cdot p_1 \cdot g^{-1}$
for the $\fg$-\/valued function~$p_1$ on~$\cEinf$ in the course of gauge 
transformation~$g\colon\alpha\mapsto\alpha^g$.
\end{proof}

Using the above lemma we describe the behaviour of homological vector field~$Q$ in the 
non\/-\/Abelian variational setup of Theorem~\ref{propNonAbelianQ}.

\begin{cor}
Under a coordinate change
\[
\alpha\mapsto\alpha' =  g\cdot \alpha \cdot g^{-1} + {d}_h g\cdot g^{-1}, \qquad
b\mapsto b'  =  g\cdot b \cdot g^{-1},
\]
where $g\in C^\infty(M^n,G)$, the variational Lie algebroid's differential~$Q$ 
is transformed accordingly
\textup{:}
\[
Q\longmapsto Q' = \vec{\partial}^{\,(\alpha')}_{[b',\alpha'] 
  + \Id_h b'} +
\tfrac12 \vec{\partial}^{\,(b')}_{[b',b']}.
\]
\end{cor}

\section{The master\/-\/functional for zero\/-\/curvature representations}\label{secMaster}
\noindent%
The correspondence between zero\/-\/curvature representations, i.e., classes of gauge\/-\/equi\-va\-lent 
solutions~$\alpha$ to the Maurer\/--\/Cartan equation, 
and non\/-\/Abelian variational Lie algebroids  
goes in parallel with the BRST\/-\/technique, 
in the frames of which ghost variables 
appear and gauge algebroids arise (see~\cite{Barnich2010,Praha2011}).
Let us therefore extend the BRST\/-\/setup of fields~$\alpha$ and ghosts~$b$
to the full BV-\/zoo of (anti)\/fields~$\alpha$ and~$\alpha^*$ and
(anti)\/ghosts~$b$ and~$b^*$ (cf.~\cite{BV1981-83,BRST,HenneauxTeitelboim}).
We note that a finite\/-\/dimensional `forefather' of what follows is discussed in detail in~\cite{AKZS}, which is devoted to $Q$-{} and $QP$-\/structures on (super)\/manifolds.
Those concepts are standard; our message is that not only the approach 
of~\cite{AKZS} to $QP$-\/structures on $G$-\/manifolds~$X$ and 
$\Pi T^*\bigl(X\times\Pi TG/G\bigr)\simeq\Pi T^*X\times\fg^*\times\Pi\fg$
remains applicable in the variational setup of jet bundles 
(i.e., whenever integrations by parts are allowed, whence many Leibniz rule
structures are lost, see Appendix~\ref{AppClassics}),
but even the explicit formulas for the BRST-\/field~$Q$ and 
the action functional~$\what{S}$ for the extended field~$\what{Q}$ 
are valid literally.
In fact, we recover the \emph{third} and \emph{fourth} equivalent formulations
of the definition for a variational Lie algebroid (cf.~\cite{AKZS,Vaintrob} or a review~\cite{YKSDB}).

Let us recall from section~\ref{secNonAbelian} that $\alpha$ is a tuple of 
even\/-\/parity fibre coordinates in the bundle $\chi\colon \Lambda^1(M^n)\otimes\fg\to 
M^n$ and $b$ are the odd\/-\/parity coordinates along fibres in the
trivial vector bundle $\Pi\xi\colon M^n\times\Pi\fg\to M^n$.
We now let all the four \emph{neighbours} of the Lie algebra $\fg$ appear on 
the stage
:
they are $\fg$ (in $\chi$), $\fg^*$, $\Pi\fg$ (in $\Pi\xi$), and $\Pi\fg^*$ 
(see~\cite{Voronov2002} and reference therein).
Let us consider the bundle $\Pi\chi^*\colon \D_1(M^n)\otimes \Pi\fg^*\to M^n$
whose fibres are dual to those in $\chi$ and also have the parity 
reversed.\footnote{In terms of~\cite{AKZS}, the Whitney sum $J^\infty(\chi)
\mathbin{\times_{M^n}} J^\infty(\Pi\chi^*)$ plays the r\^{o}le of $\Pi T^* X$ for a 
$G$-manifold~$X$; here $\fg$ is the Lie algebra  of a Lie group $G$ so that 
$\Pi\fg\simeq \Pi TG /G$.}
We denote by $\alpha^*$ the collection of odd fibre coordinates in $\Pi\chi^*$.

\begin{rem}\label{remOrder}
In what follows we do not write the (indexes for) bases of vectors in the fibres 
of  $\D_1(M^n)$ or of covectors in $\Lambda^1(M^n)$; to make the notation short,
their couplings are implicit.
Nevertheless, a summation over such ``invisible'' indexes 
in $\partial/\partial x^\mu$ and $\Id x^\nu$ is present in all formulas containing the couplings of $\alpha$ and $\alpha^*$.
We also note that $(\alpha^*)\,\overleftarrow{\Id\lefteqn{{}_h}}\phantom{{}_h}$ is a 
very interesting object because $\alpha^*$ parametrizes fibres in $\D_1(M^n)\otimes\Pi\fg^*$; 
the horizontal differential $\Id_h$
produces the forms $\Id x^i$ which are initially not 
coupled with their duals from $\D_1(M^n)$. (However, such objects cancel out in the identity $\what{Q}^2=0$, see~\eqref{eqCancelObjects}
on p.~\pageref{eqCancelObjects}.)
\end{rem}

Secondly, we consider the even\/-\/parity dual $\xi^*\colon M^n\times \fg^* \to M^n$ of 
the odd bundle $\Pi\xi$; let us denote by $b^*$ the coordinates along $\fg^*$ in the 
fibres of $\xi^*$.

Finally, we fix the ordering
\begin{equation}\label{eqFixSignsBV}
\delta\alpha \wedge \delta\alpha^* + \delta b^* \wedge \delta b
\end{equation}
of the canonically conjugate pairs of coordinates. 
By picking a volume form $\dvol(M^n)$ on the base $M^n$ we then construct
the odd Poisson bracket (variational Schouten bracket~$\lshad\,,\,\rshad$) on the senior 
$\Id_h$-\/cohomology (or \emph{horizontal} cohomology) 
space $\overline{H}^n(\chi\mathbin{\times_{M^n}}\Pi\chi^*\mathbin{\times_{M^n}}
\Pi\xi\mathbin{\times_{M^n}}\xi^*)$;
we refer to~\cite{gvbv,Dubna13} 
for a geometric theory of variations.

\begin{theor}\label{Th34}
The structure of non\/-\/Abelian variational Lie algebroid from Theorem~\textup{\ref{propNonAbelianQ}} 
is encoded on the Whitney sum $J^\infty(\chi\mathbin{{\times}_{M^n}}\Pi\chi^*\mathbin{
{\times}_{M^n}}\Pi\xi\mathbin{{\times}_{M^n}}\xi^*)$ 
of infinite jet \textup{(}super\textup{)}\/bundles by the action functional
\begin{equation*}
 \what{S} = \int \dvol(M^n) \left\{ \langle \alpha^*, [b, \alpha] + \Id_h (b) \rangle
    + \tfrac12\langle b^*, [b,b] \rangle \right\} \in 
    \overline{H}^n(\chi\mathbin{{\times}_{M^n}}\Pi\chi^*\mathbin{{\times}_{M^n}}
\Pi\xi\mathbin{{\times}_{M^n}}\xi^*)
\end{equation*}
which satisfies the classical master\/-\/equation
\[
 \lshad \what{S}, \what{S} \rshad = 0. 
\]
The functional~$\widehat{S}$ is the Hamiltonian of odd\/-\/parity evolutionary vector field~$\what{Q}$
which is defined on $J^\infty(\chi)\mathbin{{\times}_{M^n}} J^\infty(\Pi\chi^*)
\mathbin{{\times}_{M^n}} J^\infty(\Pi\xi)\mathbin{{\times}_{M^n}} J^\infty(\xi^*)$ 
by the equality
\begin{equation}\label{eqDefQHat}
 \what{Q}(\cH) \cong \lshad \what{S}, \cH \rshad
\end{equation}
for any $\cH\in\overline{H}^n(\chi\mathbin{\times_{M^n}}\Pi\chi^*\mathbin{\times_{M^n}}\Pi\xi\mathbin{\times_{M^n}}\xi^*)$.
The odd\/-\/parity field is\footnote{The referee points out that the evolutionary vector 
field~$\widehat{Q}$ is the jet\/-\/bundle upgrade of the \emph{cotangent lift} of the
field~$Q$, which is revealed by the explicit formula for the Hamiltonian~$\widehat{S}$.
Let us recall that the cotangent lift of a vector field 
${\mathcal{Q}}={\mathcal{Q}}^i\,\dd/\dd q^i$ on a (super)\/manifold~$N^m$ is the 
Hamiltonian vector field on~$T^*N^m$ given by $\widehat{{\mathcal{Q}}}={\mathcal{Q}}^i(q)\,
\dd/\dd q^i - p_j\cdot\dd{\mathcal{Q}}^j(q)/\dd q^i\,\dd/\dd p_i$; its Hamiltonian 
is~${\mathcal{S}}=p_i\,{\mathcal{Q}}^i(q)$. An example of this classical construction
is contained in the seminal paper~\cite{AKZS}.}
\begin{equation}\label{EqQHat}
  \what{Q} = \revd^{\,(\alpha)}_{[b, \alpha] + \Id_h(b)}
    + \revd^{\,(\alpha^*)}_{ (\alpha^*)\overleftarrow{\ad}^*_b}
    + \tfrac12\revd^{\,(b)}_{[b,b]}
    + \revd^{\,(b^*)}_{-\ad^*_\alpha(\alpha^*) + (\alpha^*)\overleftarrow{\Id\lefteqn{{}_h}}\phantom{{}_h} 
      + \ad^*_b(b^*) },
\end{equation}
where $\langle (\alpha^*)\,\overleftarrow{\ad}^*_b, \alpha \rangle 
  \mathrel{\stackrel{\text{\textup{def}}}{=}} \langle \alpha^*, [b, \alpha] \rangle$
and $\langle \ad^*_b (b^*), p \rangle = \langle b^*, [b, p] \rangle$ for any
$\alpha \in \Gamma(\chi)$ and $p \in \Gamma(\xi)$.
This 
evolutionary vector field is homological,
\begin{equation*}
  \what{Q}^2 = 0.
\end{equation*}
\end{theor}

\begin{proof}
In coordinates, the master\/-\/action $\what{S}=\int\widehat{\mathcal{L}}\,\dvol(M^n)$ is equal to
\begin{equation*}
\what{S} = \int \dvol (M^n) \left\{ \alpha^*_a(b^\mu
  c^a_{\mu\nu} \alpha^\nu + \Id_h(b^a)) + \tfrac12 b^*_\mu b^\beta
  c^\mu_{\beta\gamma}b^\gamma \right\};
\end{equation*}
here the summation over spatial degrees of freedom from the base $M^n$
in implicit in the horizontal differential $\Id_h$ and the respective
contractions with $\alpha^*$.
By the Jacobi identity for the variational Schouten bracket $\lshad\,,\,
\rshad$ (see~\cite{Dubna13
}), the classical master equation
$\lshad \what{S}, \what{S} \rshad = 0$ is equivalent to the homological
condition $\what{Q}^2 = 0$ for the odd-parity vector field defined
by~\eqref{eqDefQHat}.
The conventional choice of signs~\eqref{eqFixSignsBV} yields a
formula for this graded derivation,
\begin{equation*}
\what{Q} = \revd^{\,(\alpha)}_{-\rdelta \what{\cL}/\delta \alpha^*}
  + \revd^{\,(\alpha^*)}_{\rdelta \what{\cL}/\delta \alpha}
  + \revd^{\,(b)}_{\rdelta \what{\cL}/\delta b^*}
  + \revd^{\,(b^*)}_{-\rdelta \what{\cL}/\delta b},
\end{equation*}
where the arrows over $\revd$ and $\rdelta$ indicate the direction
along which the graded derivations act and graded variations are
transported (that is, from left to right and rightmost, respectively).
We explicitly obtain that\footnote{Note that $\bigl\langle\alpha^*,\overrightarrow{\Id\lefteqn{{}_h}}\phantom{{}_h}(b)\bigr\rangle\cong
-\bigl\langle(\alpha^*)\overleftarrow{\Id\lefteqn{{}_h}}\phantom{{}_h},b\bigr\rangle$ in the course of integration by parts, whence the term $
(\alpha^*_\mu)\overleftarrow{\Id\lefteqn{{}_h}}\phantom{{}_h}$ that comes from $-\rdelta \what{\cL}/\delta b^\mu$ does stand with a plus sign in the velocity of~$b^*_\mu$.}
\begin{equation*}
  \what{Q} = \revd^{\,(\alpha^a)}_{b^\mu c^a_{\mu\nu}\alpha^\nu + \Id_h(b^a)}
    + \revd^{\,(\alpha^*_\nu)}_{\alpha^*_a b^\mu c^a_{\mu\nu}}
    + \revd^{\,(b^\mu)}_{\frac12 b^\beta c^\mu_{\beta\gamma} b^\gamma}
    + \revd^{\,(b^*_\mu)}_{ \left\{ 
        -\alpha^*_a c^a_{\mu\nu} \alpha^\nu 
+ (\alpha^*_\mu)\overleftarrow{\Id\lefteqn{{}_h}}\phantom{{}_h} + b^*_a c^a_{\mu\nu} b^\nu
      \right\}}.
\end{equation*}
Actually, the proof of Theorem~\ref{propNonAbelianQ} contains the
first half of a reasoning which shows why $\what{Q}^2 =0 $. (It is
clear that the field $\what{Q}$  consists of~\eqref{eqQBRST} not
depending on $\alpha^*$ and $b^*$ and of the two new terms.)
Again, the anticommutator $[\what{Q}, \what{Q}] = 2\what{Q}^2$ is an
evolutionary vector field.
We claim that the coefficients of $\vec{\dd}/\dd \alpha^*_\nu$ and $\vec{\dd}/
\dd b^*_\mu$ in it are equal to zero.

Let us consider first the coefficient of $\vec{\dd}/\dd\alpha^*$ at the
bottom of the evolutionary derivation $\revd^{\,(\alpha^*)}_{\{\ldots\}}$ in
$\what{Q}^2$; by contracting this coefficient with $\alpha =
(\alpha^\nu)$ we obtain
\begin{equation*}
  \langle \alpha^*_a, b^\lambda c^a_{\lambda \mu} b^q c^\mu_{q\nu}
    \alpha^\nu
    - \tfrac12 b^\beta c^\mu_{\beta \gamma} b^\gamma c^{a}_{\mu\nu}
    \alpha^\nu
  \rangle. 
\end{equation*}
It is readily seen that $\alpha^*$ is here coupled with the bi\/-\/linear
skew\/-\/symmetric operator \label{pProofTh2}
$\Gamma(\xi)\times\Gamma(\xi) \to
\Gamma(\chi)$ for any fixed $\alpha\in\Gamma(\chi)$, and we show that
this operator is zero on its domain of definition.
Indeed, the comultiple  $|\,\rangle$ of $\langle \alpha^* |$ is $[b,
[b, \alpha]] - \frac12 [[b, b], \alpha]$ so that its value at any
arguments $p_1,p_2\in\Gamma(\xi)$ equals
\begin{equation*}
  [p_1, [p_2, \alpha]] - [p_2, [p_1, \alpha]] 
  - [ \tfrac12[p_1, p_2] - \tfrac12[p_2, p_1], \alpha] = 0
\end{equation*}
by the Jacobi identity.

Let us now consider the coefficient of $\vec{\dd}/\dd b^*_\mu$ 
in the vector field~$\what{Q}^2$,
\begin{multline*}
-\left[\alpha^*_{\widetilde{a}} b^{\widetilde{\mu}} c^{\widetilde{a}}_{\widetilde{\mu}a}\right]\,c^a_{\mu\nu}\alpha^\nu  
+\alpha^*_a c^a_{\mu\nu}\left[b^{\widetilde{\mu}} c^\nu_{\widetilde{\mu}\widetilde{\nu}}\alpha^{\widetilde{\nu}} + \Id_h(b^\nu)\right]
+\left(\left[\alpha^*_{\widetilde{a}} b^{\widetilde{\mu}} c^{\widetilde{a}}_{\widetilde{\mu}\mu}\right]\right)\overleftarrow{\Id\lefteqn{{}_h}}\phantom{{}_h}
\\
+\left[-\alpha^*_{\widetilde{a}} c^{\widetilde{a}}_{a\widetilde{\nu}} \alpha^{\widetilde{\nu}} + (\alpha^*_a)\overleftarrow{\Id\lefteqn{{}_h}}\phantom{{}_h} + b^*_{\widetilde{a}} c^{\widetilde{a}}_{a\widetilde{\nu}} b^{\widetilde{\nu}}\right]\, c^a_{\mu\nu} b^\nu
+b^*_a c^a_{\mu\nu}\cdot\left[\tfrac{1}{2} b^{\widetilde{\beta}} c^\nu_{\widetilde{\beta}\widetilde{\gamma}} b^{\widetilde{\gamma}}\right];
\end{multline*}
here we mark with a tilde sign those summation indexes which come from the first copy of~$\what{Q}$ 
acting from the left on $\vec{\dd}^{\,(b^*_\mu)}_{\{\ldots\}}$ in~$\what{Q}\circ\what{Q}$.
Two pairs of cancellations occur in the terms which contain the
horizontal differential~$\Id_h$. 
First, let us consider the terms in which the differential acts on~$\alpha^*$.
By contracting the index~$\mu$ with an extra copy $b=(b^\mu)$, we obtain
\begin{equation}\label{eqCancelObjects}
(\alpha^*_a)\overleftarrow{\Id\lefteqn{{}_h}}\phantom{{}_h} \,b^\lambda c^a_{\lambda\mu} b^\mu
  + (\alpha^*_a)\overleftarrow{\Id\lefteqn{{}_h}}\phantom{{}_h} c^a_{\mu\lambda} b^\lambda b^\mu.
\end{equation}
Due to the skew\/-\/symmetry of structure constants
$c^k_{ij}$ in $\fg$, at any sections $p_1,p_2\in\Gamma(\xi)$ we have that
\begin{equation*}
(\alpha^*_a)\overleftarrow{\Id\lefteqn{{}_h}}\phantom{{}_h}\cdot\left(
    p^\lambda_1 c^a_{\lambda\mu} p^\mu_2 
    - p^\lambda_2 c^a_{\lambda\mu} p^\mu_1
    + c^a_{\mu\lambda}p^\lambda_1 p^\mu_2
    - c^a_{\mu\lambda}p^\lambda_2 p^\mu_1
  \right) = 0.
\end{equation*}
Likewise, a contraction with $b=(b^\mu)$ for the 
other pair of terms with $\Id_h$, now acting on~$b$, yields
\begin{equation}\label{eqCancelObjectsAgain}
  \alpha^*_a\, c^a_{\mu\lambda}\, \Id_h(b^\lambda) b^\mu
  + \alpha^*_a\, \Id_h(b^\lambda)\, c^a_{\lambda\mu} b^\mu.
\end{equation}
At the moment of evaluation at~$p_1$ and~$p_2$, expression~\eqref{eqCancelObjectsAgain} cancels out due to the same mechanism as above.

The remaining part of the coefficient of $\vec{\dd}/\dd b^*_\mu$ in
$\what{Q}^2$ is 
\begin{multline}\label{eqTwoLines}
  - \alpha^*_z b^\lambda c^z_{\lambda a} c^a_{\mu\nu} \alpha^\nu
  + \alpha^*_z c^z_{\mu\nu} b^i c^\nu_{ij} \alpha^j 
  - \alpha^*_z c^z_{a \nu} \alpha^\nu c^a_{\mu j} b^j \\
  + b^*_\lambda c^\lambda_{a \gamma} b^\gamma c^a_{\mu j} b^j
  + b^*_{\lambda} c^\lambda_{\mu\gamma} 
    \cdot \tfrac12 b^\beta c^\gamma_{\beta \delta} b^\delta.
\end{multline}
It is obvious that the mechanisms of vanishing are different for the
first and second lines in~\eqref{eqTwoLines} whenever each of the two
is regarded as mapping which takes $b = (b^\mu)$ to a number from the field~$\Bbbk$.
Therefore, let us consider these two lines separately.

By contracting the upper line of~\eqref{eqTwoLines} with $b=(b^\mu)$,
we rewrite it as follows,
\[
\langle - \alpha^*_z, b^\lambda c^z_{\lambda a} c^a_{\mu \nu}
  \alpha^\nu b^\mu  
- c^z_{\mu\nu}b^i c^\nu_{ij}\alpha^jb^\mu
+ c^z_{a\nu} \alpha^\nu c^a_{\mu j} b^j b^\mu \rangle.
\]
Viewing the content of the co\/-\/multiple $|\,\rangle$ of $\langle
-\alpha^*|$ as bi\/-\/linear skew\/-\/symmetric mapping
$\Gamma(\xi)\times\Gamma(\xi) \to \Gamma(\chi)$, we conclude that its
value at any pair of section $p_1, p_2 \in \Gamma(\xi)$ is
\begin{align*}
{}\phantom{+} {}& [p_2, [p_1, \alpha]] - [ p_1, [p_2, \alpha]] 
  + [[p_1,  p_2], \alpha] \\
{}-{}& [p_1, [p_2, \alpha]] + [p_2, [p_1, \alpha]]
  - [[p_2, p_1], \alpha] = 0 - 0 = 0,
\end{align*}
because each line itself amounts to the Jacobi identity.

At the same time, the contraction of lower line in~\eqref{eqTwoLines} with
$b=(b^\mu)$ gives
\[
 \langle
   b^*_\lambda,
   c^\lambda_{a\gamma} b^\gamma c^a_{\mu j} b^j b^\mu
   + c^\lambda_{\mu\gamma} \cdot
     \tfrac12 b^\beta c^\gamma_{\beta \delta} b^\delta b^\mu
 \rangle.
\]
The term $|\,\rangle$ near $\langle b^*|$ determines the tri\/-\/linear
skew\/-\/symmetric mapping $\Gamma(\xi)\times\Gamma(\xi)\times\Gamma(\xi)
\to \Gamma(\xi)$ whose value at any $p_1, p_2, p_3 \in \Gamma(\xi)$
is defined by the formula
\[
\sum_{\sigma\in S_3} (-)^{\sigma}\Bigl\{
  \bigl[[ p_{\sigma(1)}, p_{\sigma(2)}], p_{\sigma(3)} \bigr]
  + \bigl[ p_{\sigma(1)}, \tfrac12 [ p_{\sigma(2)}, p_{\sigma(3)}]\bigr]
\Bigr\}.
\]
This amounts to four copies of the Jacobi identity (indeed, let us
take separate sums over even and odd permutations).
Consequently, the tri\/-\/linear operator at hand, hence the entire
coefficient of $\vec{\dd}/\dd b^*$, is equal to zero so that $\what{Q}^2=0$.
\end{proof}

\section{Gauge automorphisms of the $\widehat{Q}$-\/cohomology groups}
\noindent%
We finally describe the next generation of Lie algebroids; they arise from 
infinitesimal gauge symmetries of the quantum master\/-\/equation~\eqref{QME}
or its limit 
$\lshad \what{S},\what{S}\rshad = 0$ as $\hbar\to0$. The 
construction of infinitesimal gauge automorphisms
illustrates general principles of 
theory of differential graded Lie-{} 
or L${}_\infty$-\/algebras
(see~\cite{AKZS,KontsevichSoibelman} and~\cite{KKIgonin2003}). 

\begin{theor}\label{Th3}
An infinitesimal shift $\smash{\what{S}}\mapsto \smash{\what{S}}(\veps) = \smash{\what{S}} + \veps \lshad \smash{\what{S}},F \rshad
 + o(\veps)$, where $F$~is an odd\/-\/parity functional, 
is a gauge symmetry of the classical master\/-\/equation
$\lshad \smash{\what{S}},\smash{\what{S}}\rshad=0$.
A simultaneous shift $\eta \mapsto \eta(\veps) = \eta + \veps \lshad
\eta, F \rshad+o(\veps)$ of all functionals 
$\eta\in\overline{H}^n 
(\chi\mathbin{{\times}_{M^n}}\Pi\chi^*\mathbin{{\times}_{M^n}}\Pi\xi\mathbin{{\times}_{M^n}}\xi^*)$,
but not of the generator~$F$ itself,
preserves the structure of $\what{Q}$-\/cohomology classes.
\end{theor}

\begin{proof}
Let $F$ be an odd\/-\/parity functional and perform
the infinitesimal shift 
$\what{S} \mapsto \what{S}+ \veps\lshad \what{S}, F \rshad  + o(\veps)$
of the Hamiltonian $\what{S}$ for the differential $\what{Q}$. We have
that
\[
\lshad \smash{\what{S}}(\veps), \smash{\what{S}}(\veps) \rshad
 = \lshad \smash{\what{S}, \what{S}} \rshad 
 + 2\veps \lshad \smash{\what{S}}, \lshad \smash{\what{S}}, F \rshad \rshad
 + o(\veps).
\]
By using the shifted\/-\/graded Jacobi identity for the variational
Schouten bracket~$\lshad\,,\,\rshad$ (see~\cite{Dubna13
}) we deduce that
\[
\lshad \smash{\what{S}, \lshad \what{S}, F \rshad} \rshad =
\smash{\tfrac12} \lshad \smash{\lshad \what{S}, \what{S} \rshad, F} \rshad,
\]
so that the infinitesimal shift is a symmetry of the classical 
master\/-\/equation $\lshad \smash{\what{S}, \what{S}} \rshad = 0$. 

Now let a functional
$\eta$ mark a $\what{Q}$-cohomology class, i.e., suppose $\lshad
\what{S}, \eta  \rshad = 0$. 
In the course of simultaneous evolution $\what{S} \mapsto
\what{S}(\veps)$ for the classical master-action and $\eta \mapsto
\eta(\veps)$ for $\what{Q}$-cohomology elements, the initial condition 
$\lshad \what{S}, \eta \rshad = 0$ at $\veps = 0$ evolves as fast as
\[
\lshad \lshad \what{S}, F \rshad, \eta \rshad 
 + \lshad \what{S}, \lshad \eta, F \rshad \rshad 
= \lshad \lshad \what{S}, \eta \rshad, F \rshad = 0
\]
due to the Jacobi identity and the cocycle condition itself. 
In other words, the $\what{Q}$-\/cocycles evolve to
$\what{Q}(\veps)$-\/cocycles.

At the same time, let \emph{all} functionals 
$h\in\overline{H}^n(\chi\mathbin{{\times}_{M^n}}\Pi\chi^*\mathbin{{\times}_{M^n}}
\Pi\xi\mathbin{{\times}_{M^n}}\xi^*)$
evolve by the law $h \mapsto h(\veps) = h + \veps\lshad h, F \rshad + o(\veps)$.
Consider two representatives, $\eta$ and $\eta + \lshad \what{S}, h
\rshad$, of the $\what{Q}$-\/cohomology class for a functional~$\eta$.
On one hand, the velocity of evolution of the $\what{Q}$-\/exact term 
$\lshad \what{S}, h \rshad$ is postulated to be 
$\lshad \lshad \what{S} , h \rshad, F \rshad$;\ we claim that the
infinitesimally shif\-ted functional $\lshad \what{S}, h \rshad
(\veps)$ remains $\what{Q}(\veps)$-\/exact.
Indeed, on the other hand we have that,\ knowing the change $\what{S}
\mapsto \what{S}(\veps)$ and $h \mapsto h(\veps)$,\ the exact term's
calculated velocity is
\[
\lshad \lshad \what{S}, F \rshad, h \rshad 
 + \lshad \what{S}, \lshad h, F \rshad \rshad 
 = \lshad \lshad \what{S}, h \rshad, F \rshad 
\]
(the Jacobi identity for $\lshad\,,\,\rshad$ works again and the assertion
is valid irrespective of the parity of $h$ whenever $F$ is parity\/-\/odd).
This shows that the postulated and calculated evolutions of
$\what{Q}$-\/exact terms coincide, whence $\what{Q}$-\/coboundaries become
$\what{Q}(\veps)$-\/coboundaries after the infinitesimal shift.
\qquad
We conclude that 
the structure of $\what{Q}$-\/cohomology group stays intact
under such transformations of the space of functionals.
\end{proof}

The above picture of gauge automorphisms is extended verbatim to the full quantum setup%
\footnote{%
The Batalin\/--\/Vilkovisky differential $\boldsymbol{\Omega}^{\hbar}$ stems from the Schwinger\/--\/Dyson condition of
effective independence --~of the ghost parity-odd degrees of freedom~-- for Feynman's path integrals of the observables\,;
in earnest, the condition expresses the intuitive property $\langle1\rangle=1$ of averaging with weight factor
$\exp\left(\tfrac{\boldsymbol{i}}{\hbar}S^{\hbar}\right)$, see~\cite{BV1981-83,HenneauxTeitelboim}
and~\cite{gvbv}.},
see Fig.~\ref{FigReiterate} and~\cite[\S\,3.2]{gvbv} for detail.
\begin{figure}[htb]
{\unitlength=1mm
\special{em:linewidth 0.4pt}
\linethickness{0.4pt}
\begin{picture}(111.00,50.00)(16,0)
\put(-3,5.00){\vector(1,0){33.00}}
\put(15.00,5.00){\circle*{1.00}}
\put(-3,5.67){\makebox(0,0)[lb]{$M^n$}}
\put(13.67,1.33){\makebox(0,0)[lb]{$x$}}
\put(-2.33,10.00){\line(0,1){40.00}}
\put(-2.33,50.00){\line(1,0){31.33}}
\put(29.00,50.00){\line(0,-1){40.00}}
\put(29.00,10.00){\line(-1,0){31.33}}
\put(-1.00,43.00){\makebox(0,0)[lb]{$\overline{H}^n(\chi\mathbin{{\times}_M}\Pi\chi^*$}}
\put(1,37.00){\makebox(0,0)[lb]{${}\mathbin{{\times}_M}\Pi\xi\mathbin{{\times}_M}\xi^*)$}}
\put(-1,27.00){\makebox(0,0)[lb]{$(\bbf\cdot)\circlearrowright\bar0\mid\bar1\ni F$}}
\put(2,13.5){\makebox(0,0)[lb]{$J^{\infty}$}}
\put(10,11){\line(1,0){18}}
\put(28,11){\line(0,1){9}}
\put(28,20){\line(-1,0){18}}
\put(10,20){\line(0,-1){9}}
\put(13.33,15.5){\makebox(0,0)[lb]{$\alpha,\ \alpha^*$}}
\put(13.33,11.67){\makebox(0,0)[lb]{$b,\ b^*$}}
\put(30.67,45.00){\makebox(0,0)[lb]{$\lshad\,,\,\rshad$}}
\put(0,0){\begin{picture}(111,50)(-5,0)
\put(30,26.67){\makebox(0,0)[lb]{$F$}}
\put(28.33,21.67){\makebox(0,0)[lb]{odd}}
\put(35.00,28.00){\vector(1,0){35.00}}
\put(35.00,29.67){\makebox(0,0)[lb]%
   {$\boldsymbol{\Omega}^{\hbar}{=}{-}\boldsymbol{i}\hbar\,\Delta{+}\lshad S^{\hbar},\,\cdot\,\rshad$}}
\put(72,25.33){\makebox(0,0)[lb]{$\tfrac{\Id}{\Id\veps_F}{S}^{\hbar}$}}
\end{picture}}
\put(0,0){\begin{picture}(111,50)(-15,0)
\put(77.33,10.00){\line(1,0){31.33}}
\put(108.67,10.00){\line(0,1){40.00}}
\put(108.67,50.00){\line(-1,0){31.33}}
\put(77.33,50.00){\line(0,-1){40.00}}
\put(77.00,5.00){\vector(1,0){33.00}}
\put(93.33,5.00){\circle*{1.00}}
\put(91.67,1.00){\makebox(0,0)[lb]{$x$}}
\put(77,5.67){\makebox(0,0)[lb]{$M^n$}}
\put(75.00,12.00){\vector(0,1){13.00}}
\put(75.00,34.00){\vector(0,1){13.00}}
\put(111.00,12.00){\vector(0,1){13.00}}
\put(111.00,34.00){\vector(0,1){13.00}}
\put(113.00,45.00){\makebox(0,0)[lb]{$\left[\tfrac{\Id}{\Id\veps_{F_1}},\tfrac{\Id}{\Id\veps_{F_2}}\right]$}}
\put(78.17,43){\makebox(0,0)[lb]{$\overline{H}^n(\chi\mathbin{{\times}_M}\Pi\chi^*$}}
\put(81,37){\makebox(0,0)[lb]{${}\mathbin{{\times}_M}\Pi\xi\mathbin{{\times}_M}\xi^*)$}}
\put(79.00,31){\makebox(0,0)[lb]{$\cdot$\ quantum}}
\put(82.00,28){\makebox(0,0)[lb]{master-action}}
\put(79,24){\makebox(0,0)[lb]{$\cdot$\ observables}}
\put(89.67,20){\line(1,0){18}}
\put(107.67,20){\line(0,-1){9}}
\put(107.67,11){\line(-1,0){18}}
\put(89.67,11){\line(0,1){9}}
\put(94,15.5){\makebox(0,0)[lb]{$\alpha,\ \alpha^*$}}
\put(94,11.67){\makebox(0,0)[lb]{$b,\ b^*$}}
\put(81.67,13.5){\makebox(0,0)[lb]{$J^{\infty}$}}
\end{picture}}
\end{picture}}
\caption{The next generation of Lie algebroids: gauge automorphisms of the (quantum)
BV-\/cohomology.}\label{FigReiterate}
\end{figure}
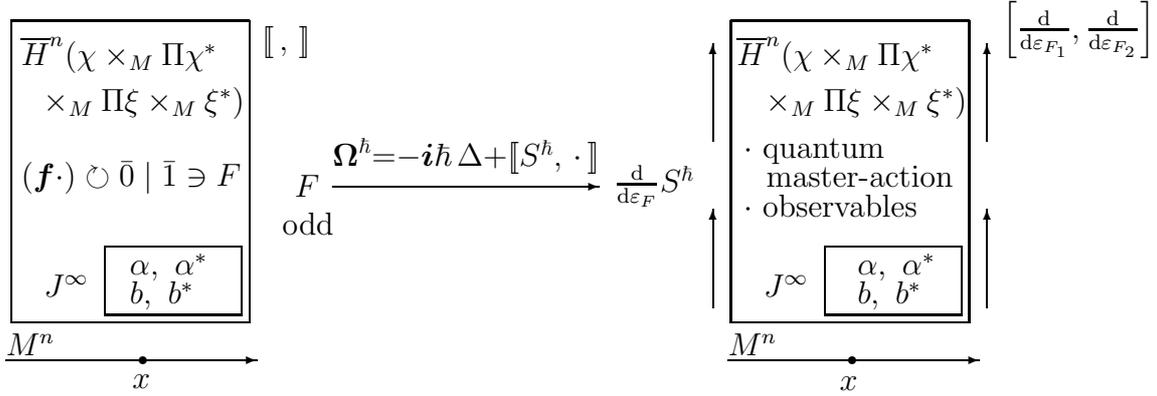
Ghost parity-odd functionals 
$F\in\ov{H}^n\left(\chi\mathbin{\times_{M^n}}\Pi\chi^*\mathbin{\times_{M^n}}\Pi\xi\mathbin{\times_{M^n}}\xi^*\right)$
are the generators of gauge transformations
$$\frac{\Id}{\Id\veps_F}S^{\hbar}=\boldsymbol{\Omega}^{\hbar}(F);$$
a parameter $\veps_F\in\BBR$ is (formally) associated with every odd functional $F$. The \emph{observables} $\bbf$ arise
through expansions $S^{\hbar}+\lambda\bbf+o(\lambda)$ of the quantum master-action\,; their evolution is given by the
coefficient
\[
\frac{\Id}{\Id\veps_F}\bbf=\lshad\bbf,F\rshad
\]
of $\lambda$ in the velocity of full action functional. It is clear also why the evolution of gauge generators $F$
--\,that belong to the domain of definition of $\boldsymbol{\Omega}^{\hbar}$ but not to its image\,--
is not discussed at all.

Let us recall from~\cite[\S\,3.2]{gvbv} and~\cite{Dubna13} that the commutator of two infinitesimal gauge transformations
with ghost parity-odd parameters, say $\cX$ and~$\cY$, is determined by the variational Schouten bracket of the two
generators\,:
\[
\left(\frac{\Id}{\Id\veps_{\cY}}\circ\frac{\Id}{\Id\veps_{\cX}}-
\frac{\Id}{\Id\veps_{\cX}}\circ\frac{\Id}{\Id\veps_{\cY}}\right)\,S^{\hbar}=
\boldsymbol{\Omega}^{\hbar}\bigl(\lshad\cX,\cY\rshad\bigr).
\]
Moreover, we discover that parity-even observables $\bbf$ play the r\^ole of ``functions'' in the world of formal products
of integral functionals. Namely, we have that
\[
\lshad\bbf\cdot\cX,\cY\rshad=\bbf\cdot\lshad\cX,\cY\rshad-(-)^{|\cX|\cdot|\cY|}\frac{\Id}{\Id\veps_{\cY}}(\bbf)\cdot\cX.
\]
In these terms, we recover the \emph{classical} notion of Lie algebroid~--- at the quantum level of horizontal cohomology
modulo $\mathop{\mathrm{im}}\Id_{h}$ in the variational setup; that classical concept is reviewed in Appendix~\ref{AppLieAlgd},
see Definition~\ref{DefLieAlgd} on p.\,\pageref{DefLieAlgd}.\ The new Lie algebroid is encoded~by
\begin{itemize}
\item
the parity\/-\/odd part of the superspace 
$\ov{H}^n(\chi\mathbin{\times_{M^n}}\Pi\chi^*\mathbin{\times_{M^n}}\Pi\xi\mathbin{\times_{M^n}}\xi^*)$
fibred over the infinite jet space for the Whitney sum of bundles 
(cf.\ Remark~\ref{RemNoExcess} on~p.\,\pageref{RemNoExcess});
\item
the quantum Batalin\/--\/Vilkovisky differential $\boldsymbol{\Omega}^{\hbar}$, which is the anchor\,;
\item
the Schouten bracket $\lshad\,,\,\rshad$, which is the Lie (super)algebra structure on the infinite-dimensional, parity-odd
homogeneity component of
$\ov{H}^n(\chi\mathbin{\times_{M^n}}\Pi\chi^*\mathbin{\times_{M^n}}\Pi\xi\mathbin{\times_{M^n}}\xi^*)$
containing the generators of gauge automorphisms in the quantum BV-model at hand.
\end{itemize}
\noindent
We see that the link between the BV-differential $\boldsymbol{\Omega}^{\hbar}$ and the classical Lie algebroid in 
Fig.\,\ref{FigReiterate} is exactly the same as the relationship between Marvan's operator $\boldsymbol{\dd}_{\alpha}$
and the non-Abelian variational Lie algebroid in Fig.\,\ref{FigNonAbelAlgd}.

Let us conclude this paper by posing an open problem of realization of the newly\/-\/built classical Lie algebroid via
the master\/-\/functional~$\EuScript{S}$ and Schouten bracket in the bi\/-\/graded, 
infinite\/-\/dimensional setup over the superbundle
of ghost parity\/-\/odd generators of gauge automorphisms (see Theorem~\ref{Th3}). We further the question to the problem of
deformation quantization in the geometry of that classical master\/-\/equation for~$\EuScript{S}$,
see~\cite{KontsevichFormality}. The difficulty which should be foreseen at once is that the
cohomological deformation technique (see~\cite{BV1981-83,KontsevichFormality} 
or~\cite{HenneauxTeitelboim,VerbKT} and references therein)
is known to be not always valid in the infinite dimension. A successful solution of the deformation quantization problem --~or its
(re-)\/iterations at higher levels, much along the lines of this paper~-- will yield the deformation parameter(s) which
would be different from~$\hbar$\,; for the Planck constant is engaged already in the picture. On the other hand, a
rigidity statement would show that there can be no deformation parameters beyond the Planck constant~$\hbar$.

\section*{Conclusion}
\noindent%
Let us sum up the geometries we are dealing with. We started with a
partial differential equation~$\cE$ for physical fields; it is possible
that $\cE$~itself was Euler\/--\/Lagrange\footnote{\label{FootEisELeq}%
  The class of admissible models is much wider than it may first seem;
  for example, the Korteweg\/--\/de Vries equation $w_t =
  -\frac12w_{xxx} + 3 ww_x$ is Euler\/--\/Lagrange with respect to the
  action functional $S_0 = \int \left\{\frac12 v_x v_t - \frac14
    v^2_{xx} - \frac12 v^3_x \right\} \Id x \wedge \Id t$ if one sets
  $w=v_x$. In absence of the model's own gauge group, its
  BV-\/realization shrinks but there remains gauge invariance in the
Maurer\/--\/Cartan equation.} 
and it could be gauge\/-\/invariant with respect to some Lie group. 
We then recalled the notion of $\fg$-\/valued zero\/-\/curvature
representations~$\alpha$ for~$\cE$; here $\fg$~is the Lie algebra of a 
given Lie group~$G$ and $\alpha$~is a flat connection's $1$-\/form in a
principal $G$-\/bundle over~$\cEinf$. By construction, this $\fg$-\/valued
horizontal form satisfies the Maurer\/--\/Cartan equation
\begin{equation}\tag{\ref{eqMC}}
\cE_{\text{MC}} = \bigl\{ \Id_h\alpha \doteq 
{\tfrac12} [\alpha, \alpha] \bigr\}
\end{equation}
by virtue of~$\cE$ and its differential consequences which constitute~$\cEinf$. 
System~\eqref{eqMC} is always gauge\/-\/invariant so that
there are linear Noether's identities~\eqref{eqNoetherId} between the
equations; if the base manifold~$M^n$ is three\/-\/dimensional, then the
Maurer\/--\/Cartan equation~$\cE_{\text{MC}}$ is Euler\/--\/Lagrange
with respect to action functional~\eqref{eqMCAction}. The main result
of this paper (see Theorem~\ref{Th34} on p.~\pageref{Th34}) is that
--\,whenever one takes not just the bundle~$\chi$ for $\fg$-\/valued
$1$-forms but the Whitney sum of four (infinite jet bundles over) vector
bundles with prototype fibers built from~$\fg$,\ $\Pi\fg$,\ $\fg^*$, 
and~$\Pi\fg^*$\,-- the gauge invariance in~\eqref{eqMC} is captured by
evolutionary vector field~\eqref{EqQHat} with Hamiltonian~$\what{S}$
that satisfies the classical master\/-\/equation~\cite{AKZS,FelderKazhdanCME12},
\begin{equation}\tag{\ref{EqCMEIntro}}
\cE_{\text{CME}} = \bigl\{
    \boldsymbol{i}\hbar\,\Delta\smash{\what{S}{\bigr|}_{\hbar=0}} = 
    \tfrac12\lshad \smash{\what{S},\what{S}} \rshad 
  \bigr\}.
\end{equation}
We notice that, by starting with the geometry of solutions to
Maurer\/--\/Cartan's equation~\eqref{eqMC}, we have constructed
another object in the category of differential graded Lie
algebras~\cite{KontsevichSoibelman}; namely, we arrive at a setup with
\emph{zero} differential $\smash{\boldsymbol{i}\hbar\,\Delta{\bigr|}_{\hbar=0}}$
and Lie (super-)\/algebra structure defined by
the variational Schouten bracket $\lshad\,,\,\rshad$. 
%
That geometry's genuine differential at $\hbar\neq0$ is given by the
Batalin\/--\/Vilkovisky Laplacian~$\Delta$ (see~\cite{BV1981-83}
and~\cite{gvbv} for its definition). 
Let us now examine whether the standard
BV-\/technique~(\cite{BV1981-83,HenneauxTeitelboim}, cf.~\cite{CattaneoMnevReshetikhin}) 
can be directly
applied to the case of zero\/-\/curvature representations, hence to
quantum inverse scattering
(\cite{SklyaninTakhtajanFaddeevTMF1979} and~\cite{Kostant1979},
also~\cite{DrinfeldICM86,
FaddeevBook1986}).

\enlargethispage{0.7\baselineskip}
It is obvious that the equations of motion~$\cE$ upon physical fields
$u = \phi(x)$ co\/-\/exist with the Maurer\/-\/Cartan equations satisfied
by zero\/-\/curvature representations~$\alpha$. 
The geometries of
non\/-\/Abelian variational Lie algebroids and gauge
algebroids~\cite{Barnich2010,Praha2011} are two manifestations of the
same construction; let us stress that the respective gauge groups can
be unrelated: there is the Lie group~$G$ for $\fg$-\/valued
zero\/-\/curvature representations~$\alpha$ and, on the other hand, there
is a gauge group (if any, see footnote~\ref{FootEisELeq}%
) for physical fields and their equations of
motion $\cE = \left\{ \delta S_0 / \delta u = 0 \right\}$.

We recalled in section~\ref{secNoetherIdMC} that the
Maurer\/--\/Cartan equation~$\cE_{\text{MC}}$ itself is
Euler\/--\/Lagrange with respect to functional~\eqref{eqMCAction} in
the class of bundles over threefolds,
cf.~\cite{TownsendAchucarro1986,AKZS,Witten1988}. 
One obtains the Batalin\/--\/Vilkovisky action by extending the geometry of
zero\/-\/curvature representations in order to capture Noether's
identities~\eqref{eqNoetherId}. It is readily seen that the required
set of Darboux variables consists of
\begin{itemize}
\item the coordinates~$\cF$ along fibres in the bundle
  $\fg^*\otimes\Lambda^2(M^3)$ for the equations~$\cE_{\text{MC}}$,
\item the \emph{antifields}~$\cF^{\dag}$ for the bundle
  $\Pi\fg\otimes\Lambda^1(M^3)$ which is dual to the
  former and which has the opposite $\BBZ_2$-\/valued ghost parity,\footnote{%
The co\/-\/multiple $|\cF\rangle$ of a $\fg$-\/valued test shift $\langle\delta\alpha|$
with respect to the $\Lambda^3(M^3)$-\/valued coupling $\langle\,,\,\rangle$ refers to~$\fg^*$
at the level of Lie algebras (i.e., regardless of the ghost parity and regardless of
any tensor products with spaces of differential forms). This attributes the 
left\/-\/hand sides of Euler\/--\/Lagrange equations $\cE_{\text{MC}}$ with 
$\fg^*\otimes\Lambda^2(M^3)$. However,
we note that the pair of canonically conjugate variables would be
$\alpha$~for $\fg\otimes\Lambda^1(M^3)$ and $\alpha^\dag$~for
$\Pi\fg^*\otimes\Lambda^2(M^3)$ whenever the Maurer\/--\/Cartan 
equations~$\cE_{\text{MC}}$ are brute\/-\/force labelled by using the
respective unknowns, that is, if the metric tensor $t_{ij}$ is not taken into account
in the coupling $\langle\delta\alpha,\cF\rangle$.} 
and also
\item the \emph{antighosts}~$b^\dag$ along fibres of
  $\fg^*\otimes\Lambda^3(M^3)$ which reproduce
  syzygies~\eqref{eqNoetherId},
  as well as
\item the \emph{ghosts}~$b$ from the dual bundle $\Pi\fg\times M^3 \to M^3$.
\end{itemize}
The standard Koszul\/--\/Tate term in the Batalin\/--\/Vilkovisky
action is then $\langle b, \boldsymbol{\dd}_{\alpha}^\dag(\alpha^\dag)\rangle$: the
classical master\/-\/action for the entire model is then\footnote{We recall that 
the Koszul\/--\/Tate component of the full BV-\/differential~$\boldsymbol{D}_{\text{BV}}$
is addressed in~\cite{VerbKT} by using the language of infinite jet bundles~---
whereas it is the BRST-\/component of~$\boldsymbol{D}_{\text{BV}}$ which we focus on
in this paper.}
\[
  (S_0 + \langle \text{BV-terms} \rangle) 
   + (S_{\text{MC}} + \langle \text{Koszul-Tate} \rangle);
\]
the respective BV-\/differentials anticommute in the Whitney sum of the
two geometries for physical fields and flat connection~$\fg$-\/forms.

The point is that Maurer\/--\/Cartan's equation~\eqref{eqMC} is
Euler\/--\/Lagrange only if~$n=3$; however, the system~$\cE_{\text{MC}}$ 
remains gauge invariant at all~$n\geqslant 2$ but
the attribution of (anti)\/fields and (anti)\/ghosts to the bundles as
above becomes \textsl{ad hoc} if~$n\neq 3$. We therefore propose to
switch from the BV-\/approach to a picture which employs the four
neighbours~$\fg$,\ $\Pi\fg$,\ $\fg^*$, and~$\Pi\fg^*$ within the
master\/-\/action $\what{S}$.
This argument is supported by the following fact~\cite{IgoninUMN}:
let~$n\geqslant 3$ for~$M^n$, suppose $\cE$~is nonoverdetermined, 
and take a finite\/-\/dimensional Lie algebra~$\fg$,
then every $\fg$-\/valued zero\/-\/curvature representation~$\alpha$ for~$\cE$
is gauge equivalent to zero
(i.e., there exists $g\in C^\infty(\cEinf,G)$ such that $\alpha = \Id_h g
\cdot g^{-1}$).
It is remarkable that Marvan's homological technique, which
contributed with the anchor~$\boldsymbol{\dd}_{\alpha}$ to our construction of
non\/-\/Abelian variational Lie algebroids, was designed for effective
inspection of the spectral parameters' (non)\/removability 
at~$n=2$ but \emph{not} in the case of
higher dimensions~$n\geqslant3$ of the base~$M^n$.

\enlargethispage{0.7\baselineskip}
We conclude that the approach to quantisation of kinematically
integrable systems is not restricted by the BV-\/technique only; 
for one can choose between the former and, e.g., flat deformation of
(structures in) equation~\eqref{EqCMEIntro} to the quantum setup
of~\eqref{QME}. It would be interesting to pursue this alternative in detail
towards the construction of quantum groups~\cite{DrinfeldICM86} and 
approach of~\cite{Kostant1979,SklyaninTakhtajanFaddeevTMF1979} to quantum
inverse scattering and quantum integrable systems. This will be the
subject of another paper.

\subsection*{Discussion}\label{pDiscussion}
Non\/-\/Abelian variational Lie algebroids which we associate with the
geometry of $\fg$-valued zero-curvature representations are the
simplest examples of such structures in a sense that the bracket
$[\,,\,]_A$ on the anchor's domain is \emph{a priori} defined in each
case by the Lie algebra $\fg$. That linear bracket is independent of
either base points $x\in M^n$ or physical fields $\phi(x)$. Another
example of equal structural complexity is given by the gauge
algebroids in Yang\/--\/Mills theory~\cite{Barnich2010}. Indeed, the
bracket $[\,,\,]_A$ on the anchor's domain of definition is then
completely determined by the multiplication table of the structure
group for the Yang\/--\/Mills field. The case of variational Poisson
algebroids~\cite{JKGolovkoVerb2008TMPh,KiselevTMPh2011} is structurally more
complex: to determine the bi-differential bracket $[\,,\,]_A$ it
suffices to know the anchor $A$; however, the bracket can explicitly
depend on the (jets of) fields or on base points. The full generality
of variational Lie algebroids setup is achieved for 2D~Toda\/-\/like
systems or gauge theories beyond Yang\/--\/Mills (e.g., for gravity).
Therefore, the objects which we describe here mediate between the
Yang\/--\/Mills and Chern\/--\/Simons models.
It is remarkable that 
``reasonable'' Chern\/--\/Simons
models can 
in retrospect narrow the class of admissible base manifolds
$M^n$ for (gauge) field theories; for the quantum objects determine topological 
invariants of threefolds (e.g., via knot theory~\cite{TuraevBook1994,
Witten1989Jones}).
Here we also admit that a triviality of the boundary conditions
is assumed by default throughout this paper (see
footnote~\ref{FootNoBoundaty} on p.~\pageref{FootNoBoundaty} and
also~\cite{AKZS}).
This is of course a model situation; a selection of ``reasonable''
geometries could in principle overload the setup with non\/-\/vanishing
boundary terms.

\appendix
\section{Lie algebroids: an overview}\label{AppClassics}
\noindent%
For consistency, let us recall the standard construction of a Lie
algebroid over a usual smooth manifold~$N^m$. By definition (see
below) it is a vector bundle $\xi \colon \Omega^{d+m} \to N^{m}$ such
that the $C^{\infty}(N)$-module $\Gamma(\xi)$ of its sections is endowed
with a Lie algebra structure $[\,,\,]_{A}$ and with an anchor,
\begin{equation}\label{eqIsoMfld}
A\in \Mor(\xi, TN) \simeq \Hom_{C^{\infty}(N)}(\Gamma(\xi), \Gamma(TN)),
\end{equation}
which satisfies Leibniz rule~\eqref{eqLeibnizInAlgd} for
$[\,,\,]_{A}$. By introducing the odd neighbour $\Pi\xi\colon \Pi\Omega
\to N$ of the vector bundle $\xi$, one represents~\cite{Vaintrob} the
Lie algebroid over $N$ in terms of an odd-parity derivation $Q$ in the
ring
$C^{\infty}(\Pi\Omega)\simeq\Gamma(\bigwedge^{\bullet}\Omega^{\ast})$ of
smooth functions on the total space $\Pi\Omega$ of the new 
superbundle~$\Pi\xi$.

Let us indicate in advance the elements of the classical definition
which are irreparably lost as soon as the base manifold becomes the
total space of an infinite jet bundle $\pi_{\infty}\colon
  J^{\infty}(\pi)\to M^n$ for a given vector bundle $\pi$ over the new
base.\footnote{
To recognize the old manifold $N^m$ in this picture and to understand
where the new bundle $\pi$ over $M^n$ stems from, one could view $N^m$
as a fibre in a locally trivial fibre bundle $\pi$ over $M^n$, so that
the new anchor takes values in $\Gamma(\pi_{\infty}^{\ast}(T\pi))$ for
the bundle induced over $J^{\infty}(\pi)$ from the tangent $T\pi$ 
to~$\pi$. It is then readily seen that the classical construction
corresponds to the special case $n=0$ and $M^n=\{\text{pt}\}$
(equivalently, one sets $\Gamma(\pi)\simeq N^m$ so that only constant
section are allowed), see Fig.~\ref{FigOldNew}. 
However, in a generic situation of non\/-\/constant
smooth sections one encounters 
differential operators 
$A\colon \Gamma(\pi^{\ast}_{\infty}(\xi))\to
\Gamma(\pi^{\ast}_{\infty}(T\pi))$ for $\xi\colon \Omega^{n+d}\to
M^n$; likewise, the `functions' standing in coefficients of all object
become differential functions of arbitrary finite order 
on~$J^{\infty}(\pi)$.%
}
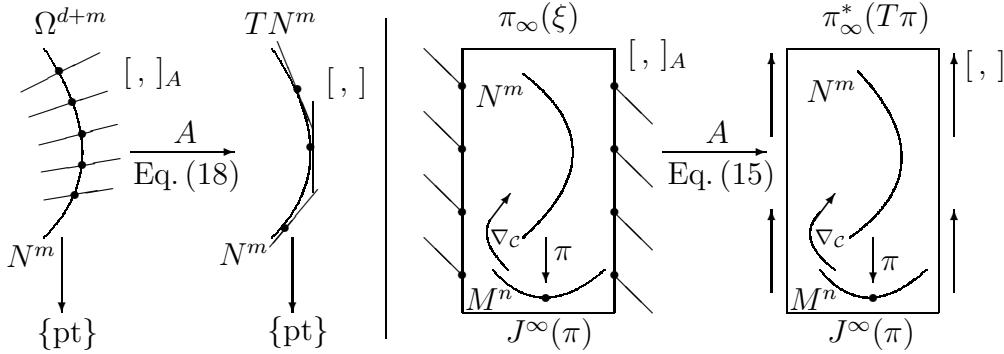
\begin{figure}[htb]
{\unitlength=1.00mm
\special{em:linewidth 0.4pt}
\linethickness{0.4pt}
\begin{picture}(130.00,44.00)
\bezier{128}(5.00,40.00)(15.00,25.00)(5.00,15.00)
\put(7.33,15.00){\vector(0,-1){10.00}}
\put(2.00,34.33){\line(2,1){10.33}}
\put(7.00,37.00){\circle*{1}}
\put(3.00,31.00){\line(3,1){10.33}}
\put(8.67,33.00){\circle*{1}}
\put(4.33,27.33){\line(4,1){10.33}}
\put(10.00,28.67){\circle*{1}}
\put(4.67,23.67){\line(6,1){10.33}}
\put(10,24.67){\circle*{1}}
\put(5.00,20.00){\line(6,1){9.00}}
\put(9,20.67){\circle*{1}}
\put(3.67,42){\makebox(0,0)[lb]{$\Omega^{d+m}$}}
\put(15.33,34.67){\makebox(0,0)[lb]{$[\,,\,]_A$}}
\put(0.00,11.00){\makebox(0,0)[lb]{$N^m$}}
\put(4,0.00){\makebox(0,0)[lb]{$\{\text{pt}\}$}}
\put(16.33,26.33){\vector(1,0){13.67}}
\put(22.00,27.5){\makebox(0,0)[lb]{$A$}}
\put(16.67,21){\makebox(0,0)[lb]{Eq.\,\eqref{eqLeibnizInAlgd}}}
\bezier{128}(35.00,40.00)(45.00,25.00)(35.00,15.00)
\put(35.67,41.00){\line(2,-5){4.67}}
\put(38.33,34.67){\circle*{1}}
\put(40.33,33.00){\line(0,-1){12.00}}
\put(40.00,27.00){\circle*{1}}
\put(41.00,21.33){\line(-5,-6){6.33}}
\put(36.67,16.33){\circle*{1}}
\put(31.33,42.00){\makebox(0,0)[lb]{$TN^m$}}
\put(42.00,33.33){\makebox(0,0)[lb]{$[\,,\,]$}}
\put(28,11.33){\makebox(0,0)[lb]{$N^m$}}
\put(37.67,15.00){\vector(0,-1){10.00}}
\put(34.33,0.33){\makebox(0,0)[lb]{$\{\text{pt}\}$}}
\put(50.00,44.00){\line(0,-1){42.33}}
\put(60.00,5.00){\line(0,1){35.00}}
\put(60.00,40.00){\line(1,0){20.00}}
\put(80.00,40.00){\line(0,-1){35.00}}
\put(80.00,5.00){\line(-1,0){20.00}}
\put(60.00,10.00){\line(-1,1){5.00}}
\put(60.00,35.00){\line(-1,1){5.00}}
\put(60.00,18.33){\line(-1,1){5.00}}
\put(60.00,26.67){\line(-1,1){5.00}}
\put(60.00,35.00){\circle*{1}}
\put(60.00,26.67){\circle*{1}}
\put(60.00,18.33){\circle*{1}}
\put(60.00,10.00){\circle*{1}}
\put(80.00,10.00){\line(1,-1){5.00}}
\put(80.00,18.33){\line(1,-1){5.00}}
\put(80.00,26.67){\line(1,-1){5.00}}
\put(80.00,35.00){\line(1,-1){5.00}}
\put(80.00,35.00){\circle*{1}}
\put(80.00,26.67){\circle*{1}}
\put(80.00,18.33){\circle*{1}}
\put(80.00,10.00){\circle*{1}}
\bezier{200}(68.00,36.67)(81,26.67)(68.00,15.00)
\bezier{88}(64,10.67)(70.33,3.33)(78.67,10.67)
\put(71,7){\circle*{1}}
\put(71,15){\vector(0,-1){6.00}}
\bezier{60}(66,10.67)(62,14.67)(63.67,17.33)
\put(63.67,17.33){\vector(3,4){2.67}}
\put(64.67,42.00){\makebox(0,0)[lb]{$\pi_{\infty}(\xi)$}}
\put(81.67,37.00){\makebox(0,0)[lb]{$[\,,\,]_A$}}
\put(61.67,32.33){\makebox(0,0)[lb]{$N^m$}}
\put(63.67,13.67){\makebox(0,0)[lb]{$\scriptstyle{\nabla_{\mathcal{C}}}$}}
\put(72.00,12.00){\makebox(0,0)[lb]{$\pi$}}
\put(60.33,5.33){\makebox(0,0)[lb]{$M^n$}}
\put(65.67,0){\makebox(0,0)[lb]{$J^{\infty}(\pi)$}}
\put(86.33,26.33){\vector(1,0){13.67}}
\put(92.00,27.5){\makebox(0,0)[lb]{$A$}}
\put(87.00,21){\makebox(0,0)[lb]{Eq.\,\eqref{EqDefFrob}}}
\put(102.67,5.00){\line(0,1){35.00}}
\put(102.67,40.00){\line(1,0){20}}
\put(122.67,40.00){\line(0,-1){35.00}}
\put(122.67,5.00){\line(-1,0){20}}
\put(100.67,7.67){\vector(0,1){11}}
\put(100.67,28.33){\vector(0,1){11}}
\put(124.67,7.67){\vector(0,1){11}}
\put(124.67,28.33){\vector(0,1){11}}
\bezier{200}(112.00,37.00)(124.00,24.33)(111.00,15.00)
\bezier{92}(107.00,10.67)(113.00,3.33)(121.00,10.67)
\put(114.00,7){\circle*{1}}
\put(114.00,15){\vector(0,-1){6}}
\bezier{60}(109.00,10.67)(103.67,14.67)(107.00,18.33)
\put(106.33,17.33){\vector(3,4){2.67}}
\put(107.33,42.00){\makebox(0,0)[lb]{$\pi^*_{\infty}(T\pi)$}}
\put(126.00,35.67){\makebox(0,0)[lb]{$[\,,\,]$}}
\put(105.25,33){\makebox(0,0)[lb]{$N^m$}}
\put(106.33,13.67){\makebox(0,0)[lb]{$\scriptstyle{\nabla_{\mathcal{C}}}$}}
\put(115.00,10.67){\makebox(0,0)[lb]{$\pi$}}
\put(103,5.33){\makebox(0,0)[lb]{$M^n$}}
\put(108.00,0){\makebox(0,0)[lb]{$J^{\infty}(\pi)$}}
\end{picture}}
\caption{From Lie algebroids $\bigl(\xi,A,[\,,\,]_A\bigr)$ to variational Lie algebroids
$\bigl(\pi_\infty^*(\xi),A,[\,,\,]_A\bigr)$.}\label{FigOldNew}
\end{figure}
The new anchor almost always becomes a positive order operator in
total derivatives; it takes values in the space of
$\pi_{\infty}$-vertical, evolutionary vector fields that preserve the
Cartan distribution on $J^{\infty}(\pi)$. But Newton's binomial
formula for the derivatives in $A$ prescribes that the old
identification $A(f\cdot\cX) = f\cdot A(\cX)$ of the
two module structures for $\Gamma(\Omega)\ni\cX$ and
$\Gamma(TN)$ is no longer valid (and isomorphism~\eqref{eqIsoMfld} is
lost). Simultaneously, Leibniz rule~\eqref{eqLeibnizInAlgd} is not valid,
e.g., even if one takes $A = \id$ for $\xi = \pi$.

To resolve the arising obstructions, for the new definition of a
\emph{variational} Lie algebroid over $J^{\infty}(\pi)$  we take the proven
Frobenius property, 
\begin{equation}\label{EqDefFrob}
[\img A, \img A] \subseteq \img A,
\end{equation}
of the anchor to be the Lie algebra homomorphism $\bigl(\Gamma\Omega,
[\,,\,]_{A}\bigr)\to\bigl(\Gamma(TN),[\,,\,]\bigr)$. In other words, we
postulate an implication but not the initial hypothesis of classical
construction. Such resolution was proposed in~\cite{KiselevTMPh2011}
for the (gra\-ded-)\/commutative setup of Poisson geometry on
$J^{\infty}(\pi)$ or for the geometry of 2D Toda-like systems and
BV-formalism for gauge-invariant models such as the Yang-Mills
equation (see~\cite{Praha2011} and also~\cite{Barnich2010} in which an
attempt to recognize the classical picture is made in a manifestly
jet-bundle setup). In this paper we show that the new approach is
equally well applicable in the non-Abelian case of Lie algebra\/-\/valued
zero-curvature representations for partial differential equations
$\cEinf\subseteq J^{\infty}(\pi)$
(which could offer new insights in the arising gauge cohomology
theories~\cite{Marvan2002}). 

\subsection{The classical construction of a Lie algebroid}\label{AppLieAlgd}
Let $N^m$ be a smooth real $m$-\/dimensional manifold 
($1\leq m\leq+\infty$)
and denote by $\cF=C^\infty(N^m)$ the ring of smooth functions on it.
The space $\varkappa=\Gamma(TN)$ of sections of the tangent 
bundle~$TN$ is an $\cF$-\/module. Simultaneously, the space~$\varkappa$
is endowed with the natural Lie algebra structure~$[\,,\,]$ 
which is the commutator of vector fields,
\begin{equation}\label{Commutator}
[X,Y]=X\circ Y - Y\circ X, \qquad X,Y\in\Gamma(TN).
\end{equation}
As usual, we regard the tangent bundle's 
sections as first order differential operators with zero free term.

The $\cF 
$-\/module structure of the space $\Gamma(TN)$ manifests itself for the generators of~$\varkappa$ through the Leibniz rule,
\begin{equation}\label{LeibnizVF}
[f\,X,Y]=(f\,X)\circ Y-f\cdot Y\circ X-Y(f)\cdot X,\qquad 
   f\in\cF. 
\end{equation}
The coefficient $-Y(f)$ of the vector field~$X$ in the last term 
of~\eqref{LeibnizVF} belongs again to the prescribed ring~$\cF$.

Let $\xi\colon\Omega^{m+d}\to N^m$ be another vector bundle over~$N$ and suppose that its fibres are $d$-\/dimensional. Again, the space~$\Gamma\Omega$ of sections of the bundle~$\xi$ is a module over the ring~$\cF$ 
of smooth functions on the manifold~$N^m$. 

\begin{define}[\cite{Vaintrob}]\label{DefLieAlgd}
A \textsl{Lie algebroid} over a 
manifold $N^m$ is a vector bundle $\xi\colon\Omega^{d+m}\to N^m$ 
whose space of sections $\Gamma\Omega$ 
is equipped with a Lie algebra
structure $[\,,\,]_A$ together with a bundle morphism 
$A\colon\Omega\to TN$,
called the \textsl{anchor}, such that the Leibniz rule
\begin{equation}\label{eqLeibnizInAlgd}
[f\cdot\cX,\cY]_A=f\cdot[\cX,\cY]_A
  -\bigl(A(\cY)f\bigr)\cdot\cX
\end{equation}
holds for any
$\cX,\cY\in\Gamma\Omega$ and any $f\in C^\infty(N^m)$. 
\end{define}

\begin{example}
Lie algebras are toy examples of Lie algebroids over a point. The other standard examples are the tangent bundle and the Poisson algebroid structure of the cotangent bundle to a Poisson manifold~\cite{YKSMagri}.
\end{example}

\begin{lemma}[\cite{Herz}
]\label{PropMorphismFollows}
The anchor $A$ maps the bracket $[\,,\,]_A$ for sections
of the vector bundle~$\xi$ 
to the Lie bracket $[\,,\,]$ for sections of the tangent bundle
to the manifold $N^m$.
\end{lemma}

This property is a consequence
of Leibniz rule~\eqref{eqLeibnizInAlgd} and the Jacobi identity
for the Lie algebra structure $[\,,\,]_A$ in~$\Gamma\Omega$.
Remarkably, the assertion of Lemma \ref{PropMorphismFollows} is often
\textsl{postulated} (for convenience, rather than derived) 
as a part of the definition of a Lie algebroid, \textrm{e.\,g.}, 
see~\cite{Vaintrob,Voronov2002} vs \cite{Herz,YKSMagri}. 

In the course of transition from usual manifolds $N^m$ to jet spaces
$J^{\infty}(\pi)$ it is natural that maps of spaces of sections become
nonnegative\/-\/order linear differential operators. For example, the anchors
will be operators in total derivatives
$A\in\CDiff\bigl(\Gamma(\pi^*_{\infty}(\xi)) \to
\Gamma(\pi^*_{\infty}(T\pi))\bigr)$ for spaces of sections of induced
vector bundles; note that the $\pi_{\infty}$-\/vertical component of the
tangent bundle to $J^{\infty}(\pi)$ is the target space.\footnote{We recall that
both junior and senior Hamiltonian differential operators have \emph{positive}
differential orders for all Drinfel'd\/--\/Sokolov hierarchies associated with
the root systems; this construction yields a class of variational Poisson algebroids.
The anchors which are linear operators of zero differential order are a rare exception
(however, see~\cite{TMPh2005} in this context).}
Whenever that differential order is strictly positive, one loses the
property of $A$ to be a homomorphism over the algebra $\cF(\pi) =
C^{\infty}(J^{\infty}(\pi))$ of differential functions of arbitrary
finite order. Indeed, consider the first\/-\/order anchor
$\boldsymbol{\partial}_{\alpha} = [\cdot, \alpha] + \Id_h$, which we discuss in this
paper (cf.~\cite{Marvan2002}): even though $[f\cdot p, \alpha] =
f\cdot [p, \alpha]$, the horizontal differential $\Id_h$ acts by the
Leibniz rule so that $\boldsymbol{\partial}_{\alpha} (f \cdot p) \neq
f\cdot \boldsymbol{\partial}_{\alpha}(p)$ if $f \neq \const$. We see that such map
of horizontal module of sections for a bundle $\pi^{*}_{\infty}(\xi)$
induced over $J^{\infty}(\pi)$ is not completely determined by the
images of a basis of local sections in $\xi$, which is in contrast with
the classical case in~\eqref{eqIsoMfld}.

Likewise, the Leibniz rule expressed by~\eqref{eqLeibnizInAlgd} does not
hold whenever a section $\cY\in\Gamma(\pi^{*}_{\infty}(\xi))\simeq
\Gamma(\xi)\mathbin{\otimes_{C^{\infty}(M)}}C^{\infty}(J^{\infty}(\pi))$
contains derivatives $u_\sigma$ of fibre coordinates~$u$ in~$\pi$. 
A (counter)\/example is as follows: take $\xi = T\pi$ and set $A =
\id\colon \bigl(\Gamma(\pi^{*}_{\infty}(T\pi)), [\,,\,]_A\bigr)\to
\bigl(\Gamma(\pi^{*}_{\infty}(T\pi)), [\,,\,]\bigr)$, where both Lie algebra
structures are the commutator of evolutionary vector fields. 
Let $\cX,\cY\in\Gamma\bigl(\pi^{*}_{\infty}(T\pi)\bigr)$ and $f\in
C^{\infty}(J^{\infty}(\pi))$. Then we have that
\begin{multline*}
[f\,\cX, \cY]_A = \partial^{(u)}_{f\,\cX}(\cY)
- \partial^{(u)}_{\cY}(f\cdot \cX) = f\cdot [\cX, \cY]_A -
A(\cY)(f)\cdot\cX \\
+ \sum_{|\sigma|>0}\: \sum_{\substack{\rho \cup \tau = \sigma\\
    |\rho|>0}} \frac{\Id^{|\rho|}}{\Id
  x^{\rho}}(f)\cdot\frac{\Id^{|\tau|}}{\Id x^{\tau}}(\cX)\cdot
  \frac{\partial}{\partial u_{\sigma}}(\cY).
\end{multline*}
As soon as the above two ingredients of the classical definition are
lost, we take for definition of an anchor in a \emph{variational} Lie
algebroid over $J^{\infty}(\pi)$ the involutivity $[\img A, \img A]
\subseteq \img A$ of image of a linear operator $A \in
\CDiff\bigl(\Gamma(\pi^{*}_{\infty}(\xi)), \Gamma(\pi^{*}_{\infty}(T\pi))\bigr)$
whose values belong to the space of generating sections of
evolutionary vector fields on $J^{\infty}(\pi)$ (alternatively, the
anchor could take values in a smaller Lie algebra of infinitesimal
symmetries for a given equation $\cEinf\subseteq
J^{\infty}(\pi)$). Notice that the anchor is then a Lie algebra
homomorphism by construction, namely, $A\colon
\bigl(\Gamma(\pi^{*}_{\infty}(\xi)), [\,,\,]_A\bigr) \to \bigl(\varkappa(\pi),
[\,,\,]\bigr)$. Note further that, on one hand, the bracket $[\,,\,]_A$ could be
induced on $\Gamma(\pi^{*}_{\infty}(\xi))/\ker A$ by the property
$[A(p_1),A(p_2)] = A([p_1,p_2]_A)$ of commutation closure for the image of~$A$. 
(Such is the geometry of Liouville\/-\/type Toda\/-\/like systems or the BRST-{} and BV-\/approach to gauge field models,
see~\cite{Praha2011,SymToda2009,KiselevTMPh2011} and references therein). On the other hand, the bracket $[\,,\,]_A$ can be present \textit{ab initio} in the
picture: such is the case of Hamiltonian operators~$A$ in the Poisson
formalism or the geometry of zero\/-\/curvature representations
(indeed, we then have $[\,,\,]_A = [\,,\,]_{\fg}$ for the Lie algebra~$\fg$ of a gauge group~$G$). This alternative yields four natural examples of
variational Lie algebroids.


\subsection{The odd neighbour $\Pi\xi\colon\Pi\Omega\to N^m$ and
  differential $Q^2=0$.}\label{AppOdd}
The odd neighbour of a vector bundle $\xi\colon \Omega^{m+d}\to N^m$
over a smooth real manifold $N^m$ is the vector bundle
$\Pi\xi\colon\Pi\Omega^{m+d}\to N^m$ over the same base and with
the same vector space $\BBR^d$ taken as the prototype for the fibre
over each point $x\in\mathcal{U}_\alpha\subseteq N^m$: the coordinate
diffeomorphism is $\varphi_\alpha\colon \mathcal{U}_\alpha \times
\BBR^d \to \Pi\Omega^{d+m}$. Moreover, the topology of the bundle
$\Pi\xi$ coincides with that of $\xi$ so that the gluing
transformations $g^\Pi_{\alpha\beta}\in GL(d,\BBR)$ in the fibres
over intersections $\mathcal{U}_\alpha \cap \mathcal{U}_\beta
\subseteq N^m$ of charts, smoothly depending on $x\in
\mathcal{U}_\alpha \cap \mathcal{U}_\beta$, are exactly the same as
the fibres' reparametrizations $g_{\alpha\beta}(x)$ in the bundle
$\xi$. However, notice that these \emph{linear} mapping can not feel
any grading of the object which they transform (in particular,
$g_{\alpha\beta}$ can not grasp the $\BBZ_2$-valued parity of such
$\BBR^d$); this indifference is the key element in a construction of
the odd neighbour. Namely, let the coordinates $b^1,\ldots,b^d$ along
the fibres $(\Pi\xi)^{-1}(x)\simeq \BBR^d$ be $\BBZ_2$-parity
odd,\footnote{%
The parity reversion $\Pi\colon p \leftrightarrows b$ acts on the
fibre coordinates but not on a basis $\vec{e}_i$ in $\BBR^d$. To keep
track of a distinction between the two geometries, we formally denote
by $\mathfrak{e}_i = \Pi\vec{e}_i$ the basis in $\BBR^d$ which referes
to the $\BBZ_2$-graded setup.}
i.e., introduce the $\BBZ_2$-grading $|{\cdot}|\colon x^i \mapsto \bar{0},\
b^j \mapsto \bar{1}$ for the ring of smooth $\BBR$-\/valued functions on
the total space $\Pi\Omega$ of the superbundle (the grading then acts by a
multiplicative group 
homomorphism $|{\cdot}|\colon C^{\infty}(\Pi\Omega)\to \BBZ_2$). 
We have that
$C^{\infty}(\Pi\Omega)\simeq\Gamma(\bigwedge^{\bullet}\Omega^{*})$,
where $\Omega^{*}$ denotes the space of fibrewise\/-\/linear functions 
on~$\Omega$. By construction, the new space of graded coordinate
functions on $\Pi\Omega$ is an $\BBR$-\/algebra and a
$C^{\infty}(N)$-\/module.

Notice further that the space of the bundle's sections in principle
stays intact; however, it is not the sections of $\Pi\xi$ which will be
explicitly dealt with in what follows but it is a convenient handling
of cochains and cochain maps for $\Gamma(\xi)$ by coding those objects
and structures in terms of fibrewise-homogeneous functions on~$\Pi\Omega$.

\begin{rem}
It is important to distinguish between sections
$\bp\in\Gamma(\xi)$, $\bp\colon N^m \to
\Omega^{m+d}$, and fibre coordinates $p^j$ on the total space $\Omega$
of the vector bundle $\xi$. Indeed, $\partial p^j / \partial x^i
\equiv 0$ by definition whereas the value at $x\in N^m$ of a
derivative $\tfrac{\partial}{\partial x^i}(\bp^j)(x)$ of a
section $\bp$ could be any number. In particular, consider
the Jacobi identity for the Lie algebra structure $[\,,\,]_A \colon
\Gamma(\xi)\times\Gamma(\xi) \to \Gamma(\xi)$ in a Lie algebroid. 
Let $\bp_\mu = \bp_\mu^i \vec{e}_i$ be sections of $\xi$,  
here $\mu=1$,\ $2$,\ $3$, and denote by $c^k_{ij}(x)$ the values at
$x\in N^m$ of the structure constants of $[\,,\,]_A$ with respect to a
natural basis $\vec{e}_i$ of local sections. Then we have that
\begin{multline}\label{eqJFull}
0 = \sum_{\circlearrowright} [[\bp_1, \bp_2]_A,
\bp_3]_A = \sum_{\circlearrowright}[
\bp^i_1 c^k_{ij}(x)\bp^j_2\cdot \vec{e}_k,\bp^\ell_3\cdot\vec{e}_\ell]_A\\
= \sum_{\circlearrowright}
\bp^i_1\bp^j_2\bp^\ell_3\cdot\Bigl\{c^k_{ij}(x)c^n_{kl}(x)\cdot \vec{e}_n -
 \bigl(A{\bigr|}_{x}(\vec{e}_\ell)\bigr) (c^k_{ij}(x)) \cdot \vec{e}_k 
 \Bigr\} \\ 
+ \sum_{\circlearrowright} c^k_{ij}(x)\cdot \Bigl\{\bp^i_1\bp^j_2\cdot
  \bigl(A{\bigr|}_x(\vec{e}_k)\bigr)(\bp^\ell_3)(x)\cdot\vec{e}_\ell 
 - \bp^j_2\bp^\ell_3\cdot \bigl(A{\bigr|}_x(\vec{e}_\ell)\bigr)(\bp^i_1)(x)\cdot\vec{e}_k \\
 - \bp^i_1\bp^\ell_3\cdot \bigl(A{\bigr|}_x(\vec{e}_\ell)\bigr)(\bp^j_2)(x)\cdot\vec{e}_k
 \Bigr\}.
\end{multline}
Clearly, if the coefficients $\bp^i_\mu$ are viewed as local
coordinates along fibres in $\Omega$ over $
{x}\in N^m$
parametrized by $x^1,\ldots, x^m$, then the vector fields
$A(\vec{e}_\ell)\in\Gamma(TN)$ no longer act on such $p^i_\mu$'s so that
the entire last sum in~\eqref{eqJFull} vanishes.

We refer to~\cite{Praha2011,KiselevTMPh2011} for a discussion on the immanent
presence and recovery of the 'standard,' vanishing terms in the course of
transition $C^{\infty}(\Pi\Omega)\to C^{\infty}(\Omega)\to
\Alt\bigl(\Gamma(\xi)\times\cdots\times\Gamma(\xi)\to\Gamma(\xi)\bigr)$ from homogeneous functions of the odd fibre coordinates to
$\Gamma(\xi)$-\/valued cochains (and cochain maps such as the Lie
algebroid differential $\Id_{A}$). A detailed analysis of properties
and interrelations between the four neighbours $\fg$, $\Pi\fg$,
$\fg^*$, and $\Pi \fg^*$ is performed in~\cite{Voronov2002} 
(here $m=0$, $N^m=\{\text{pt}\}$, and the Lie
algebroid~$\Omega$ is a Lie algebra~$\fg$).
\end{rem}

\begin{prop}[\cite{Vaintrob}]\label{propVariationalQ}
The Lie algebroid structure on $\Omega$ is encoded by the homological
vector field $Q$ on $\Pi\Omega$, i.e., by a 
derivation in the ring
$C^{\infty}(\Pi\Omega) = \Gamma(\bigwedge^{\bullet}\Omega^{*})$,
\[
 Q = A^{\alpha}_i(x)\, b^i \frac{\partial}{\partial x^{\alpha}} -
 \tfrac12 b^i c^k_{ij}(x)\, b^j \frac{\partial}{\partial b^k}, \qquad
 [Q,Q] =0 \quad  \Longleftrightarrow \quad 2Q^2 =0,
\]
where
\begin{itemize}
\item
$(x^\alpha)$ is a system of local coordinates near a point $x\in N^m$,
\item
$(p^i)$ are local coordinates along the $d$-\/dimensional fibres of $\Omega$
and $(b^i)$ are the respective coordinates on $\Pi\Omega$, and
\item
the formula $[\vec{e}_i,\vec{e}_j]_A=c^k_{ij}(x)\,\vec{e}_k$ gives
the structure constants for a $d$-\/element
local basis $(\vec{e}_i)$ of sections
in $\Gamma\Omega$ over the point~$x$, and 
$A(\vec{e}_i)=A_i^\alpha(x)\cdot\dd/\dd x^\alpha$ is the image of $\vec{e}_i$
under the anchor~$A$. 
\end{itemize}
\end{prop}

\begin{proof}[Sketch of the proof]
The reasoning goes in parallel with the proof of
Theorem~\ref{propNonAbelianQ}. First, we recall that the anchor $A
= \|A^\alpha_i\|^{1\leqslant\alpha\leqslant m}_{1\leqslant i \leqslant
  d}$ is the Lie algebra homomorphism by
Lemma~\ref{PropMorphismFollows}. Second, we note that the homogeneous
(in odd-parity coordinates $b^j$) coefficients of $\partial / \partial
b^k$, $1\leqslant k \leqslant d$, in $Q^2$ encode the tri-linear,
totally skew-symmetric map $\omega_3\colon
\Gamma(\xi)\times\Gamma(\xi)\times\Gamma(\xi) \to \Gamma(\xi)$ whose value at any
$p_1$,\ $p_2$,\ $p_3 \in \Gamma(\xi)$ is twice the right-hand side of
Jacobi's identity~\eqref{eqJFull}. Here we use the fact that cyclic
permutations of three objects are even (in terms of permutation's
$\BBZ_2$-parity), whence it is legitimate to extend the summation
$\sum_{\circlearrowright}$ to a sum over the entire permutation group
$S_3$:
\[
\sum_\circlearrowright \omega_3(p_{\sigma(1)}, p_{\sigma(2)},
p_{\sigma(3)}) = \frac12 \sum_{\sigma \in S_3} (-)^{\sigma} \omega_3(p_{\sigma(1)}, p_{\sigma(2)},
p_{\sigma(3)}).
\]
The presence of zero section in the left-hand side of Jacobi
identity~\eqref{eqJFull} implies that the respective coefficient of
$\partial/\partial {\boldsymbol{b}}$ in $Q$ vanishes.\footnote{%
Notice that the second step of this reasoning is simplified further in
the case of non\/-\/Abelian variational Lie algebroids (see
p.~\pageref{propNonAbelianQ}) because in that case the bracket $[\,,\,]_A$ is a
given Lie algebra structure in~$\fg$; it is described globally by
using the structure constants $c^k_{ij}$ regardless of the base
manifold.}
\end{proof}

\begin{rem}
The coefficient $+\tfrac12$ in the homological evolutionary vector
field~$Q$ in Theorem~\ref{propNonAbelianQ}, but not the opposite
value $-\tfrac12$ in the canonical formula (see
Proposition~\ref{propVariationalQ} above) is due to our choice of sign
in a notation for the zero\/-\/curvature representation $\alpha = A_i\cdot
\Id x^i$: one sets either $\Psi_{x^i} + A_i\Psi = 0$ or $\Psi_{x^i} =
A_i\Psi$ for the wave function $\Psi$. The second option is adopted by
repetition but it tells us that the gauge connection's $\fg$-\/valued 
one\/-\/form is minus~$\alpha$.
\end{rem}

\begin{rem}
The correspondence $f_k \leftrightarrow \omega_k$ between homogeneous
functions $f_k(x;\boldsymbol{b},\ldots, \boldsymbol{b})\in C^{\infty}(\Pi\Omega)$ on the total
space of the superbundle $\Pi\xi$ and $k$-chain maps $\omega_k \colon
\Gamma(\xi)\times\cdots\times\Gamma(\xi) \to C^{\infty}(N)$
correlates 
the homological vector field $Q$ with the Lie algebroid differential $\Id_A$ that acts by the standard Cartan formula. Namely, the following diagram is commutative,
\[
\begin{CD}
\Id_{A} \colon @. \omega_k @>>> \omega_{k+1}\\
              @.  @VVV          @VVV \\
Q \colon      @.  f_k     @>>>  f_{k+1}.
\end{CD}
\]
The wedge product of $k$- and $\ell$-chains corresponds under the
vertical arrows of this diagram to the ordinary $\BBZ_2$-graded
multiplication of the respective functions from~$C^{\infty}(\Pi\Omega)$.

The main examples of this construction are the de Rham differential on a
manifold $N^m$  (as before, set $\xi:=\pi$ and let $A = \id$), the
Chevalley\/--\/Eilenberg differential for a Lie algebra~$\fg$ (let $m=0$,
$N^m= \{\text{pt}\}$, and take $(\Omega, [\,,\,]_A) = (\fg, [\,,\,]_{\fg})$
and $A = 0$), and the de~Rham differential on a symplectic manifold
(here $\xi\colon \Lambda^1(N^m)\to N^m$, $A= \lshad P, \cdot
\rshad$ is the Poisson differential given by a bi-vector $P$
satisfying $\lshad P, P \rshad = 0$ and having the inverse
symplectic two-form $P^{-1}$, and $[\,,\,]_A$ is the
Koszul\/--\/Dorfman\/--\/Daletsky\/--\/Karas\"{e}v bracket~\cite{Dorfman,YKSMagri}).

The Hamiltonian homological evolutionary vector field~$Q$ that encodes the
\emph{variational} Poisson algebroid structure over a jet space
$J^{\infty}(\pi)$ was de facto written
in~\cite{JKGolovkoVerb2008TMPh}. The BRST-\/differential~$\boldsymbol{Q}$ is another example of such construction over jet
spaces $J^{\infty}(\pi)\supseteq \cEinf$ containing the Euler\/--\/Lagrange equations for gauge\/-\/invariant models.
\end{rem}

\subsection*{Acknowledgements}
The authors are grateful to the anonymous referee for helpful suggestions and remarks.
The authors thank the organizing committee of the conference
`Nonlinear Mathematical Physics: Twenty Years of~JNMP' (Nordfj{\o}rdeid,
Norway, 2013) for partial support. The work of A.V.K.\ was supported in part by
JBI~RUG project~103511 (Groningen);
A.O.K.\ was supported by ISPU scholarship for young scientists.
A part of this research was done while A.V.K.\ was visiting 
at the $\smash{\text{IH\'ES}}$ (Bures\/-\/sur\/-\/Yvette); 
the financial support and hospitality of this institution are gratefully acknowledged.


\begin{thebibliography}{99}\normalsize

\bibitem{TownsendAchucarro1986}
\by{Achucarro~A., Townsend~P.~K.} (1986)
A Chern\/--\/Simons action for three\/-\/dimensional anti\/-\/de Sitter
supergravity theories, \jour{Phys.\ Lett.} \vol{B180}:1--2, 89--92.

\bibitem{AKZS}
\by{Alexandrov M., Schwarz A., Zaboronsky O., Kontsevich M.} (1997)
The geometry of the master equation and topological quantum field
theory, \jour{Int.~J.\ Modern Phys.} \vol{A12}:7, 
1405--1429.\ \texttt{arXiv:hep-th/9502010}

\bibitem{Barnich2010}
\by{Barnich G.} (2010)
A note on gauge systems from the point of view of Lie algebroids, AIP Conf.\ Proc.~\vol{1307}
XXIX Workshop on Geometric Methods in Physics (June~27 -- July~3, 2010; Bia\l owie\.za, Poland),
7--18.\ \texttt{arXiv:1010.0899} [math-ph]

\bibitem{BV1981-83}
\by{Batalin I., Vilkovisky G.} (1981)
Gauge algebra and quantization, \jour{Phys.\ Lett.} \vol{B102}:1, 27--31;\\
\quad
\by{Batalin I. A., Vilkovisky G. A.} (1983)
Quantization of gauge theories with linearly dependent generators, \jour{Phys.\ Rev.} \vol{D29}:10, 2567--2582. 

\bibitem{BRST}
\by{Becchi C., Rouet A., Stora R.} (1976) Renormalization of gauge theories, \jour{Ann.\ Phys.} \vol{98}:2, 287--321;\\
\quad
\by{Tyutin I.~V.} (1975) Gauge invariance in field theory and statistical mechanics, \jour{Preprint} Lebedev FIAN no.~39.

\bibitem{BVV}
\by{Bocharov A. V., Chetverikov V. N., Duzhin S. V. \textit{et al}.} (1999)
\book{Symmetries and conservation laws for differential equations of mathematical physics}
(I.~S.~Krasil'shchik and A.~M.~Vinogradov, eds.)
Transl.\ Math.\ Monographs~\vol{182}, AMS, Providence~RI. 

\bibitem{Bourbaki}
\by{Bourbaki N.} (1989) 
\book{Lie Groups and Lie Algebras -- Chapters 1--3}, Springer, Berlin. 

\bibitem{CattaneoMnevReshetikhin}
\by{Cattaneo A. S., Mn\"ev P., Reshetikhin M.}
Semiclassical quantization of classical field theories,
\jour{Preprint} \texttt{arXiv:1311.2490} [math-ph], 36~pp.

\bibitem{Dorfman}
\by{Dorfman I. Ya.} (1993) \book{Dirac structures and integrability of
nonlinear evolution equations},
J.~Whiley \& Sons, Chichester.

\bibitem{DrinfeldICM86}
\by{Drinfel'd V. G.} (1986)
Quantum groups, 
\jour{Zap.\ Nauchn.\ Sem.\ Leningrad.\ Otdel.\ Mat.\ Inst.\ Steklov.\ (LOMI)}~\vol{155} Differentsialnaya Geometriya, Gruppy Li i Mekh.~VIII, 18--49, 193 (Russian); 
\jour{J.~Soviet Math.} (1988) \vol{41}:2, 898--915; 
Proc.\ Int.\ Congr.\ Math.~\vol{1, 2} (Berkeley CA, 1986), 
AMS, Providence RI (1987), 798--820.

\bibitem{DNF}
\by{Dubrovin B. A, Fomenko A. T., Novikov S. P.} (1991)
\book{Modern Geometry -- Methods and Applications. Part~I: The Geometry of
Surfaces, Transformation Groups, and Fields}, 
Grad.\ Texts in Math.~\vol{93}, Springer, Berlin.


\bibitem{FaddeevBook1986}
\by{Faddeev L. D., Takhtajan L. A.} (1987)
\book{Hamiltonian methods in the theory of solitons},
Springer Ser.\ Soviet Math., Springer\/-\/Verlag, Berlin.

\bibitem{FelderKazhdanCME12}
\by{Felder G., Kazhdan D.} (2012) The classical master equation,
\jour{Preprint} \texttt{arXiv:1212.1631} [math.AG], 59~pp.


\bibitem{JKGolovkoVerb2008TMPh}
\by{Golovko V. A., Krasil'shchik I. S., Verbovetsky A. M.} (2008)
Variational Poisson\/--\/Nijenhuis structures for partial differential
equations,
\jour{Theor.\ Math.\ Phys.} \vol{154}:2, 227--239.\ \texttt{arXiv:0812.4684} [math.DG]

\bibitem{HenneauxTeitelboim}
\by{Henneaux~M., Teitelboim~C.} (1992)
\book{Quantization of gauge systems}, Princeton University Press, Princeton~NJ.

\bibitem{Herz}
\by{Herz J.-C.} (1953) Pseudo\/-\/alg\`ebres de Lie.~I, II.
\jour{C.~R.~Acad.\ Sci.\ Paris} \vol{236}, 1935--1937,
2289--2291.

\bibitem{IgoninUMN}
\by{Igonin S. A.} (2003)
Horizontal cohomology with coefficients, and nonlinear zero\/-\/curvature representations, 
\jour{Russ.\ Math.\ Surv.}~\vol{58}:1, 180--182.

\bibitem{KKIgonin2003}
\by{Igonin S., Kersten P. H. M., Krasil'shchik I. S.} (2003)
On symmetries and cohomological invariants of equations possessing
flat representations,
\jour{Differential Geom.\ Appl.} \vol{19}, 319--342.\ \texttt{arXiv:math.DG/0301344}


\bibitem{GDE2012}
\by{Kiselev A. V.} (2012)
The twelve lectures in the (non)\/commutative geometry of differential
equations,
{\it Preprint} $\smash{\text{IH\'ES}}$/M/12/13, 140~pp.

\bibitem{gvbv}
\by{Kiselev A. V.} (2013) The geometry of variations in Batalin\/--\/Vilkovisky formalism, 
\jour{J.~Phys.\ Conf.\ Ser.} \vol{474},
Proc. XXI~Int.\ Conf.\ `Integrable systems \& quantum symmetries' 
(June 12--16, 2013; CVUT Prague, Czech Republic), 012024, 51~pp.\ %
\texttt{arXiv:1312.1262} [math-ph]

\bibitem{Dubna13}
\by{Kiselev A. V.} (2013) The Jacobi identity for graded\/-\/commutative variational 
Schouten bracket revisited, \jour{Preprint} \texttt{arXiv:1312.4140} [math-ph], 7~pp.


\bibitem{Praha2011}
\by{Kiselev A. V.} (2012)
Homological evolutionary vector fields in Korteweg\/--\/de Vries,
Liouville, Maxwell, and several other models, 
\jour{J.~Phys.\ Conf.\ Ser.} \vol{343}, Proc.\ 7th Int.\ workshop QTS-7
`Quantum Theory and Symmetries' (August 7--13, 2011; CVUT Prague,
Czech Republic), 012058, 20~pp.\ %
\texttt{arXiv:1111.3272} [math-ph]

\bibitem{TMPh2005}
\by{Kiselev A. V.} (2005) Hamiltonian flows on Euler\/-\/type equations,
\jour{Theor.\ Math.\ Phys.} \vol{144}:1, 952--960.\ %
\texttt{arXiv:nlin.SI/0409061}


\bibitem{SymToda2009}
\by{Kiselev A. V., van de Leur J. W.} (2010) 
Symmetry algebras of Lagrangian
Liouville\/-\/type systems, \jour{Theor.\ Math.\ Phys.} 
\vol{162}:3, 149--162. 
\texttt{arXiv:0902.3624} [nlin.SI]

\bibitem{KiselevTMPh2011}
\by{Kiselev A. V., van de Leur J. W.} (2011)
Variational Lie algebroids and homological evolutionary vector fields,
\jour{Theor.\ Math.\ Phys.} \vol{167}:3, 772--784.\ %
\texttt{arXiv:1006.4227} [math.DG]

\bibitem{KiselevLeur2011}
\by{Kiselev A. V., van de Leur J. W.} (2011)
Involutive distributions of operator-valued
evolutionary vector fields and their affine geometry, 
Proc.\ 5th Int.\ workshop `Group
analysis of differential equations and integrable systems' 
(June 6--10, 2010; Protaras, Cyprus), 99--109.
\texttt{arXiv:0904.1555} [math-ph]

\bibitem{RingersProtaras}
\by{Kiselev~A.~V., Ringers~S.} (2013)
A comparison of definitions for the Schouten bracket on jet spaces,
Proc.\ 6th Int.\ workshop `Group
analysis of differential equations and integrable systems' 
(June 18--20, 2012; Protaras, Cyprus), 127--141.
\texttt{arXiv:1208.6196} [math.DG]


\bibitem{KontsevichFormality}
\by{Kontsevich M.} (2003)
Deformation quantization of Poisson manifolds.~I,
\jour{Lett.\ Math.\ Phys.} \vol{66}, 157--216.\ %
\texttt{arXiv:q-alg/9709040}

\bibitem{KontsevichSoibelman}
\by{Kontsevich~M., Soibelman~Y.} (2009)
Notes on $A_\infty$-\/algebras, $A_\infty$-\/categories and
non\/-\/commutative geometry.
\book{Homological Mirror Symmetry: New Developments and Perspectives}
(A.~Kapustin \textit{et al.}, eds).
Lect.\ Notes in Physics \vol{757}, Springer, Berlin\/--\/Heidelberg, 153--219.

\bibitem{YKSDB}
\by{Kosmann\/-\/Schwarzbach Y.} (2004)
Derived brackets, \jour{Lett.\ Math.\ Phys.} \vol{69}, 61--87.

\bibitem{YKSMagri}
\by{Kosmann\/-\/Schwarzbach Y., Magri F.} (1990)
Poisson\/--\/Nijenhuis structures,
\jour{Ann.\ Inst.\ H.~Poin\-car\'e\textup{,} ser.~A\textup{:} 
Phys.\ Th\'eor.} \vol{53}:1, 35--81.

\bibitem{Kostant1979}
\by{Kostant B.} (1979) Quantization and representation theory,
\book{Representation theory of Lie groups} (G.~L.~Luke, ed.)
Proc.\ SRC/LMS Research Symposium (Oxford, 
1977),
London Math.\ Soc.\ Lect.\ Note Ser.~\vol{34}, Cambridge Univ.\ Press, 
Cambridge\/--\/New York, 287--316.



\bibitem{Marvan2002}
\by{Marvan M.} (2002)
On the horizontal gauge cohomology and non\/-\/removability of the
spectral parameter, \jour{Acta Appl.\ Math.} \vol{72}, 51--65.


\bibitem{Olver}
\by{Olver P. J.} (1993) 
\book{Applications of Lie groups to differential equations},
Grad.\ Texts in Math.\ \vol{107} (2nd ed.), Springer\/--\/Verlag,~NY.

\bibitem{SklyaninTakhtajanFaddeevTMF1979}
\by{Skljanin E. K., Tahtad\v{z}jan L. A., Faddeev L. D.} (1979)
Quantum inverse problem method.~I.
\jour{Teoret.\ Mat.\ Fiz.}~\vol{40}:2, 194--220. 

\bibitem{TuraevBook1994}
\by{Turaev V. G.} (1994)
\book{Quantum invariants of knots and $3$-\/manifolds},
de Gruyter Stud.\ in Math.~\vol{18}, Walter de Gruyter \& Co., Berlin.

\bibitem{Vaintrob}
\by{Vaintrob A. Yu.} (1997)
Lie algebroids and homological vector fields,
\jour{Russ.\ Math.\ Surv.} \vol{52}:2, 428--429.

\bibitem{VerbKT}
\by{Verbovetsky A. M.} (2002) Remarks on two approaches to horizontal cohomology:
compatibility complex and the Koszul\/--\/Tate resolution,
\jour{Acta Appl.\ Math.} \vol{72}:1--2, 123--131.\ %
\texttt{arXiv:math.DG/0105207}

\bibitem{Voronov2002}
\by{Voronov T.} (2002) Graded manifolds and Drinfeld doubles for Lie bialgebroids. \book{Quantization, Poisson brackets, and beyond}
(T.~Voronov, ed.) Contemp.\ Math.\ \vol{315}, AMS, Providence~RI,
131--168.\ \texttt{arXiv:math.DG/0105237}

\bibitem{Witten1988}
\by{Witten~E.} (1988)
Topological sigma models,
\jour{Commun.\ Math.\ Phys.} \vol{118}:3, 411--449.

\bibitem{Witten1989Jones}
\by{Witten E.} (1989)
Quantum field theory and the Jones polynomial,
\jour{Commun.\ Math.\ Phys.} \vol{121}:3, 351--399.

\bibitem{ZSh}
\by{Zakharov V. E., Shabat A. B.} (1979)
Integration of nonlinear equations of mathematical physics by the
method of inverse scattering.~II,
\jour{Functional Analysis and its Applications}
\vol{13}:3, 166--174. 

\end{thebibliography}
\end{document}